\documentclass[svgnames]{amsart}
\usepackage[english]{babel}
\usepackage[utf8]{inputenc}

\usepackage{svg}
\usepackage{hyperref}
\hypersetup{colorlinks,citecolor=Black,linkcolor=Black,urlcolor=Black}
\usepackage{url}

\usepackage{amssymb} \usepackage{amsfonts} \usepackage{amsmath}
\usepackage{amsthm} \usepackage{epsfig, subfig} \usepackage{mathabx}
\usepackage{ amscd, amsxtra, latexsym,pb-diagram,pb-xy}
\usepackage{tikz}\usetikzlibrary{cd}
\usepackage{bbm}

\usepackage[shortlabels]{enumitem}

\theoremstyle{plain}
\newtheorem{thm}{Theorem}[section] 
\newtheorem{lem}[thm]{Lemma}
\newtheorem{claim}[thm]{Claim}
\newtheorem{prop}[thm]{Proposition}
\newtheorem{cor}[thm]{Corollary}
\newtheorem{thmintro}{Theorem}

\theoremstyle{definition}
\newtheorem{ass}{Assumption}

\newtheorem{defn}[thm]{Definition}
\newtheorem{rem}[thm]{Remark}
\newtheorem{conv}[thm]{Convention}
\newtheorem{ex}[thm]{Example}
\newtheorem{questintro}[thmintro]{Question}

\newcommand{\frakS}{\mathfrak S}

\newcommand{\C}{\mathcal C}
\newcommand{\CS}{\mathcal C S}
\newcommand{\CSN}{\mathcal C S/N}
\newcommand{\ov}{\overline}

\newcommand{\dist}{\mathrm{d}}
\newcommand{\nest}{\sqsubseteq}

\newcommand{\orth}{\bot}
\newcommand{\transverse}{\pitchfork}
\newcommand{\MCG}{\mathcal{MCG}^\pm (S)}

\newcommand{\link}{\text{Lk}}
\newcommand{\Cay}[2]{\operatorname{Cay}\left(#1,#2\right)}
\newcommand{\Supp}{\operatorname{Supp}}
\newcommand{\Aut}{\operatorname{Aut}}
\newcommand{\Out}{\operatorname{Out}}
\newcommand{\Comm}{\operatorname{Comm}}

\title[QI-rigidity of random quotients of MCG]{Random quotients of mapping class groups are quasi-isometrically rigid}

\author[G. Mangioni]{Giorgio Mangioni}
    \address{(Giorgio Mangioni) Maxwell Institute and Department of Mathematics, Heriot-Watt University, Edinburgh, UK}
    \email{gm2070@hw.ac.uk}

\begin{document}

\begin{abstract}
    We prove several rigidity properties for random quotients of mapping class groups of  surfaces, namely whose kernel is normally generated by the $n$th steps of finitely many independent random walks. Firstly, we generalise a celebrated theorem of Ivanov's: every automorphism of the corresponding quotient of the curve graph is induced by a mapping class. Next, we show that, if a finitely generated group is quasi-isometric to a random quotient, then the two groups are weakly commensurable. This uses techniques from the world of hierarchically hyperbolic groups: indeed, in the process we clarify a proof of Behrstock, Hagen, and Sisto on the quasi-isometric rigidity of mapping class groups, which might possibly be applied to other hierarchically hyperbolic groups. Finally, we show that the automorphisms groups of our quotients, as well as their abstract commensurators, coincide with the groups themselves. 

    Our results hold for a wider family of quotients, namely those whose kernel act by sufficiently large translations on the curve graph. This class also includes quotients by suitable powers of a pseudo-Anosov element.
\end{abstract}

\maketitle


\section*{Introduction}
One of the motivating perspectives on mapping class groups of orientable, finite-type surfaces has long been the analogy with higher rank lattices. One notable similarity in their behaviour is that both classes exhibit strong quasi-isometric rigidity properties: if $G$ is in one of these families and $H$ is a group quasi-isometric to $G$, then $G$ and $H$ are weakly commensurable (in the sense of Theorem~\ref{thmintro:qi_for_random} below). For mapping class groups, this result relies on Ivanov's theorem that every automorphism of the curve graph is induced by a mapping class \cite{Ivanov:autC, Korkmaz}.

In recent years there has been a growing interest in studying quotients and subgroups of mapping class groups, in order to establish the analogues of Ivanov's theorem and then deduce similar quasi-isometric rigidity properties (see e.g. \cite{Brendle_Margalit, McLeay} for a wide class of normal subgroups, and, among others, \cite{mangioni2023rigidity} for quotients by suitable powers of Dehn Twists).

In this paper, we address the problem for \emph{random quotients} of mapping class groups, which we now define. Given an acylindrically hyperbolic group $G$, let $\mu$ be a \emph{permissible} probability measure, in the sense of Definition~\ref{defn:permissible}. A \emph{random walk} on $G$ with respect to $\mu$ is a sequence $\{X_n\}_{n\in\mathbb{N}}$ of random variables, taking values in $G$, such that:
\begin{itemize}
    \item $P(X_0=e)=1$, i.e. the process starts at the identity element of $G$;
    \item For every $n\ge 1$, $P(X_n=hs\mid X_{n-1}=h)=\mu(s)$, i.e. the step between $X_{n-1}$ and $X_n$ is chosen according to the distribution $\mu$.
\end{itemize}
The reader should keep in mind the case where $G$ has a finite, symmetric generating set $S$, and $\mu$ is the uniform probability on $S$. In this setting, the $n$th step $X_n$ of the random walk is a word of $n$ letters, each chosen with uniform probability from the alphabet $S$.

Now let $\{X^1_n,\ldots, X^k_n\}_{n\in\mathbb{N}}$ be independent random walks on $G$ with respect to permissible probability measures. Let $N_n\colon eq \langle\langle X^1_n,\ldots, X^k_n\rangle\rangle,$ and call $G/N_n$ a \emph{random quotient} of $G$. A property $\mathcal P$ holds \emph{asymptotically almost surely} (a.a.s) if the probability that $G/N_n$ satisfies $\mathcal P$ approaches $1$ as the step $n$ goes to infinity. Heuristically, $G/N_n$ should a.a.s. exhibit the properties of a “generic” quotient of $G$ by $k$ “independent” relators.

\subsection*{Rigidity results}
The first rigidity theorem we prove is the analogue of Ivanov's theorem. Below, $\MCG$ denotes the \emph{extended} mapping class group, where we allow orientation-reversing mapping classes, and $\C S$ is the curve graph. To avoid technicalities coming from small surfaces, we state this and the following results for surfaces of high enough complexity.

\begin{thmintro}[{Combinatorial rigidity}]\label{thmintro:comb_for_random}
    Let $S$ be a surface of complexity at least $4$, and let $\{X^1_n, \ldots, X^k_n\}_{n\in\mathbb{N}}$ be $k$ random walks with respect to permissible probability measures. For every $k\in\mathbb{N}_{>0}$, the map $\MCG/{N_n}\to \Aut(\CS/{N_n})$, induced by the natural action of $\MCG/{N_n}$ on $\CS/{N_n}$, is  a.a.s. an isomorphism.
\end{thmintro}
%

\noindent For our second result, recall that two groups $G$ and $H$ are \emph{weakly commensurable} if there exist two finite-index subgroups $G'\le G$ and $H'\le H$ and two finite normal subgroups $L\unlhd G'$ and $M\unlhd H'$ such that $G'/L\cong H'/M$.

\begin{thmintro}[{Quasi-isometric rigidity}]\label{thmintro:qi_for_random}
    Let $S$ be a surface of complexity at least $4$, and let $\{X^1_n, \ldots, X^k_n\}_{n\in\mathbb{N}}$ be $k$ random walks with respect to permissible probability measures. For every $k\in\mathbb{N}_{>0}$ the following holds asymptotically almost surely:
    \begin{enumerate}
        \item\label{item:weak_qi_rig} Every quasi-isometry $f\colon  \MCG/N_n\to \MCG/N_n$ is within finite distance from the left multiplication by some element $g\in \MCG/N_n$.
        \item\label{item:commensurable_for_rand} If a finitely generated group $G$ and $\MCG/N_n$ are quasi-isometric then they are weakly commensurable. 
    \end{enumerate}
\end{thmintro}
\noindent We stress that the probability that~\eqref{item:commensurable_for_rand} holds does not depend on $G$, but only on $n$ and $k$. Indeed,~\eqref{item:commensurable_for_rand} follows from a ``quantitative'' version of \eqref{item:weak_qi_rig}, together with Proposition \ref{prop:qi_quantitativa} which is a general criterion for quasi-isometric rigidity of a group and is therefore of independent interest.

Finally, random quotients exhibit a strong form of algebraic rigidity, which generalises the analogue results of Ivanov's for mapping class groups \cite{Ivanov:autC}:

\begin{thmintro}[Algebraic rigidity]\label{thmintro:alg_for_rand}
Let $S$ be a surface of complexity at least $4$, and let $\{X^1_n, \ldots, X^k_n\}_{n\in\mathbb{N}}$ be $k$ random walks with respect to permissible probability measures. For every $k\in\mathbb{N}_{>0}$, the following holds asymptotically almost surely:
\begin{enumerate}
\item The map $\MCG/N_n\to \Aut(\MCG/N_n)$, which maps every element to the corresponding inner automorphism, is an isomorphism. In particular, $\Out(\MCG/N_n)$ is trivial.
\item The abstract commensurator of $\MCG/N_n$ coincides with $\MCG/N_n$ itself. In particular, every isomorphism between finite index subgroups of $\MCG/N_n$ is the restriction of an inner automorphism.
\end{enumerate}
\end{thmintro}

\subsection*{Large translation quotients}
Random quotients of mapping class groups are an instance of \emph{large translation quotients}, which are those whose kernel $N\unlhd \MCG$ acts on the curve graph with sufficiently large \emph{minimum translation length} (see Definition \ref{defn:lrq}). For these quotients, the natural map $\CS\to \CSN$ has positive injectivity radius, and is therefore a covering.

We list here some examples of large translation quotients:
\begin{enumerate}[label=(\roman*)]
    \item\label{item:rq} Random quotients are a.a.s large translation quotients, in view of Lemma~\ref{lem:random_are_LT}.
    \item The mapping class group itself, seen as the quotient by the trivial subgroup, is trivially a large translation quotient.
    \item Given a pseudo-Anosov element $g\in \MCG$, there exists $K\in\mathbb{N}_{>0}$ such that $\langle\langle g^K\rangle\rangle$ has large translation (see Corollary~\ref{cor:pA_large}: this is not obvious a priori, as the whole normal closure might in principle contain elements of small translation length).
    \item\label{item:hemb} Generalising the previous example, let $H\le \MCG$ be a \emph{hierarchically hyperbolically embedded} subgroup, in the sense of \cite[Definition 6.1]{hhs_asdim}. Then the normal closure of a suitable finite-index subgroup of $H$ has large translation (see Lemma \ref{lem:fixing_hhg_for_hhe} for details).
\end{enumerate}

\noindent Theorems~\ref{thmintro:comb_for_random} to~\ref{thmintro:alg_for_rand} then follow from the analogue results for large translation quotients, which we now state. Again, to avoid technicalities we restrict to surfaces of large complexity.

\begin{thmintro}[{see Theorem~\ref{thm:ivanov}}]\label{thmintro:comb}
    Let $S$ be a surface of complexity at least $4$, and let $N\unlhd\MCG$ be a normal subgroup. If the minimum translation length of $N$ is at least $9$, then $\MCG/N\cong \Aut(\CSN)$ via the natural action.
\end{thmintro}
\noindent In Subsection~\ref{subsec:ivanov_sporadic}, we show that a similar result also holds for most surfaces of low complexity, though the natural homomorphism might only be surjective with finite kernel (this happens already for mapping class group \cite{Korkmaz,Luo}).
 \par\medskip

The analogues of Theorems~\ref{thmintro:qi_for_random} and~\ref{thmintro:alg_for_rand} hold for those large translation quotients with a suitable \emph{hierarchically hyperbolic group} (HHG) structure, in the sense of \cite{HHSII}. Roughly, a group $G$ is a HHG if there exists a collection of hyperbolic spaces $\{\C U\}_{U\in \frakS}$ together with coarsely Lipschitz ``coordinate projections'' $\pi_U\colon  G\to \C U$ satisfying some properties. An element $U\in\frakS$ is called a \emph{domain}, and $\C U$ is its \emph{coordinate space}.  Mapping class groups of surfaces are the motivating example of HHGs, with projections given by from subsurface projections to curve graphs of subsurfaces (see e.g. \cite{HHSI}, though most of the results date back to the seminal work of Masur and Minsky \cite{MM_1, MM_2}). 

The large translation quotients we consider have a HHG structure which is somewhat ``inherited'' from the natural structure on $\MCG$ (see Convention \ref{conv:surface-inherited}). Roughly, we ask that domains of the structure are $N$-orbits of surfaces, that the top-level coordinate space coincides with $\CSN$, and that the coordinate space of the orbit of any proper subsurface $U\subsetneq S$ is quasi-isometric to the original curve graph of $U$. Remarkably, all quotients of type \ref{item:rq} to \ref{item:hemb} satisfy these assumptions.

\begin{thmintro}[{see Theorem~\ref{thm:qirigid}}]\label{thmintro:qi}
    Let $S$ be a surface of complexity at least $4$, and let $N\unlhd\MCG$ be a normal subgroup. If the minimum translation length of $N$ is at least $9$ and $\MCG/N$ has a surface-inherited HHG structure, then the following holds:
    \begin{enumerate}
     \item Every quasi-isometry $f\colon  \MCG/N\to \MCG/N$ is within finite distance from a left multiplication.
    \item If a finitely generated group $G$ and $\MCG/N$ are quasi-isometric then they are weakly commensurable.
    \end{enumerate}
\end{thmintro}

\begin{thmintro}[{see Theorem~\ref{thm:Out}}]\label{thmintro:alg}
Let $S$ be a surface of complexity at least $4$, and let $N\unlhd\MCG$ be a normal subgroup. If the minimum translation length of $N$ is at least $9$ and $\MCG/N$ has a surface-inherited HHG structure, then $$\MCG/N\cong \Aut(\MCG/N)
\cong \Comm(\MCG/N),$$ where $\Comm(\MCG/N)$ is the abstract commensurator.
\end{thmintro}

\noindent We point out that, in Theorems~\ref{thmintro:comb} to \ref{thmintro:alg}, the bound on the minimum translation length is explicit and uniform over all surfaces. Though it might possibly be sharpened, in all notable applications one is often allowed to replace $N$ with a suitable subgroup to enlarge the minimum translation length.

\subsection*{Sketch of proofs}
\subsubsection*{Combinatorial rigidity}
To show that any automorphism $\ov\phi\colon \CSN\to\CSN$ is induced by some element $\ov g\in\MCG/N$ we proceed as follows. First, we find an automorphism $\phi:\CS\to\CS$ which “lifts” $\ov \phi$, i.e. that makes the following diagram commute, where $\pi\colon \CS\to \CSN$ is the quotient map:

$$\begin{tikzcd}
    \CS\ar{r}{\phi}\ar{d}{\pi}&\CS\ar{d}{\pi}\\
    \CSN\ar{r}{\ov \phi}&\CSN
\end{tikzcd}$$

The key observation is that, if we regard both $\CS$ and $\CSN$ as simplicial \emph{complexes}, then $\pi$ is a covering map (see Lemma \ref{lem:covering}). Moreover, by a result of Harer \cite{Harer:CS_wedge}, $\CS$ is simply connected whenever $S$ has large enough complexity, and we can lift $\ov\phi$ by standard arguments of covering theory. Then one can apply Ivanov's theorem to show that the lift $\phi$ is induced by a mapping class $g\in\MCG$, and therefore its image $\ov g\in\MCG/N$ induces $\ov \phi$.

\subsubsection*{Quasi-isometric rigidity}
We follow a general strategy, developed by Behrstock, Hagen, and Sisto in \cite{quasiflat}, to prove quasi-isometric rigidity of certain HHG. The key idea from that paper is that any self-quasi-isometry $f$ of a HHG, satisfying some additional assumptions, induces an automorphism of a certain graph, which encodes the intersection patterns of maximal \emph{quasiflats} (that is, quasi-isometric embeddings of $\mathbb{R}^n$ of maximum dimension). In the case of $\MCG/N$, such graph is precisely the quotient of the curve graph, since maximal quasiflats of $\MCG/N$ morally correspond to subgroups generated by maximal families of commuting Dehn Twists. But then, by combinatorial rigidity, the automorphism of $\CSN$ is itself induced by some element $\ov g\in\MCG/N$, and with a little more effort one can show that $f$ coarsely coincides with the left multiplication by $\ov g$.

In \cite{quasiflat}, the main result to extract combinatorial data from a self-quasi-isometry of a HHG is \cite[Theorem 5.7]{quasiflat}, which is then used in \cite[Theorem 5.10]{quasiflat} to give a new proof of quasi-isometric rigidity of mapping class groups. Unfortunately, as pointed out by Jason Behrstock after that paper was published, \cite[Theorem 5.7]{quasiflat} does not apply to mapping class groups. Indeed, if $U\subsetneq S$ is a proper subsurface which is not an annulus, then any maximal collection of pairwise disjoint subsurfaces that contains $U$ must also contain the boundary annuli of $U$. Hence, mapping class groups do not satisfy \cite[Assumption 2]{quasiflat}, which roughly states that such a $U$ should be the only subsurface shared by two maximal collections of pairwise disjoint subsurfaces. We stress that \cite[Theorem 5.7]{quasiflat} is true as stated, and it has been used in \cite{VeechII, mangioni2023rigidity} to study quasi-isometric rigidity of extensions of Veech groups and Dehn twist quotients of mapping class groups, respectively.
 
Behrstock, Hagen, and Sisto, worked out a modification of their Assumptions which allows them to prove a version of \cite[Theorem 5.7]{quasiflat} which applies to the mapping class group.  Rather than writing their argument separately, they instead opted to allow me to incorporate a version of those assumptions into this work, and I am extremely grateful for this possibility. Then the proof of \cite[Theorem 5.10]{quasiflat} runs verbatim once one replaces \cite[Theorem 5.7]{quasiflat} with our Proposition \ref{lem:automorphism_of_frakS_G}, whose Assumptions \eqref{ass1} - \eqref{ass6} are satisfied by both $\MCG$ and our quotients (see Proposition \ref{lem:ass_satisfied}).

\subsubsection*{Algebraic rigidity}
By \cite[Corollary 14.4]{HHSI} $\MCG/N$ acts acylindrically on the top-level coordinate space of the HHG structure, which coincides with $\CSN$ for quotients with surface-inherited structures. Now, an isomorphism $\phi\colon H\to H'$ between finite-index subgroups of $\MCG/N$ can be seen as a self-quasi-isometry of $\MCG/N$, and therefore coarsely coincides with the left-multiplication by some element $\ov g\in \MCG/N$. One can then try to prove that the conjugation by $\ov g$ restricts to the given isomorphism on $H$. We do just that in Theorem \ref{thm:Out}, using tools from \cite{Commensurating} for acylindrical actions on hyperbolic spaces. Notice that finite normal subgroups could cause the outer automorphism group to be finite rather than trivial, so the key technical results we need is that $\MCG/N$ does not contain non-trivial finite normal subgroups, Lemma \ref{lem:nofinite}.

\subsubsection*{What is new for $\MCG$} The above results were already well-established for $\MCG$, which is trivially a large translation quotient. We summarise here the (partial) novelty of our approach.
\begin{itemize}
    \item Our proof of Theorem~\ref{thmintro:comb} does not yield a new argument for Ivanov's theorem, as it heavily relies on it.
    \item Theorem~\ref{thmintro:qi} provides a new proof of quasi-isometric rigidity of mapping class groups, first proven in \cite{BKMM}, by refining a previous attempt of Behrstock, Hagen, and Sisto from \cite{quasiflat} (see Corollary \ref{cor;qirig_mcg}). We expect that suitable adaptations of the tools from Section \ref{sec:qusiiso_of_HHG} can be used to classify quasi-isometries of other hierarchically hyperbolic groups.
    \item Finally, the original proof of Theorem~\ref{thmintro:alg} from \cite{Ivanov:autC} relies on an algebraic characterisation of powers of Dehn twists. Instead, we derive algebraic rigidity from quasi-isometric rigidity, using the machinery from \cite{Commensurating} to study automorphisms of acylindrically hyperbolic groups. To the knowledge of the author, this type of approach first appeared in a recent paper of Sisto and the author \cite{mangioni2023rigidity}, and can possibly be applied beyond the setting of mapping class groups.
\end{itemize}

\subsection*{Comparison with quotients by large powers of Dehn twists}
In view of \cite{mangioni2023rigidity}, very similar rigidity results hold for quotients of the form $\MCG/DT_K$, where $DT_K$ is the normal subgroup generated by all $K$-th powers of Dehn Twists (at least when $S$ is a punctured sphere, and conjecturally for all surfaces). Such quotients can be regarded as “Dehn filling quotients” of mapping class groups, as pointed out and explored in \cite{dfdt} and then in \cite{BHMS}, where they are proven to be hierarchically hyperbolic. 

While the proofs from \cite{mangioni2023rigidity} also rely on lifting, the projection $\CS\to \CS/DT_K$ is not a covering map, and indeed it is quite far from being locally injective; therefore lifting properties follow from different technologies, involving results from \cite{dahmani:rotating} about rotating families and the existence of finite rigid sets inside the curve graph, as defined in \cite{AL}. Moreover, our approach to quasi-isometric rigidity requires Dehn twist flats to survive in the quotient, while they disappear in $\MCG/DT_K$ for any choice of $K$. In particular, quasi-isometries of $\MCG/DT_K$ induce automorphisms of a graph which is not $\CS/DT_K$, and then one has to relate these two graphs with further combinatorial considerations.

\subsection*{Outline of the paper}
In Section \ref{sec:setting} we define large translation quotients and develop the lifting tools that are used throughout the paper. 

In Section \ref{sec:combrig} we prove combinatorial rigidity for the general case, which is Theorem \ref{thm:ivanov}, and for some surfaces of low complexity, see Subsection \ref{subsec:ivanov_sporadic}. This establishes Theorem \ref{thmintro:comb}.

In Section \ref{sec:qusiiso_of_HHG} we show that a self-quasi-isometry of a hierarchically hyperbolic group satisfying certain properties, namely Assumption \eqref{ass1} - \eqref{ass6}, induces an automorphism of a certain graph, encoding the intersections of certain maximal quasiflats (see Proposition \ref{lem:automorphism_of_frakS_G}). Then in Section \ref{sec:qirigid} we specialise this to those large translation quotients of mapping class groups with a suitable HHG structure (see Convention \ref{conv:surface-inherited}). We relate the graph from Proposition \ref{lem:automorphism_of_frakS_G} to $\CSN$, and use the combinatorial rigidity results from Section \ref{sec:combrig} to produce an element of $\MCG/N$ whose left-multiplication is within finite distance from a given quasi-isometry. This proves Theorem \ref{thmintro:qi} (see Theorem \ref{thm:qirigid}).

In Section \ref{sec:algrig} we combine quasi-isometric rigidity and some tools from \cite{Commensurating} about acylindrical actions on hyperbolic spaces to show that any automorphism between finite index subgroups of $\MCG/N$ is the restriction of the conjugation by a unique element $\ov g\in\MCG/N$ (see Theorem \ref{thm:Out} for the existence of such an element, and Lemma \ref{lem:aut_inj} for the uniqueness). Therefore, the outer automorphism group of $\MCG/N$, as well as its abstract commensurator, are the smallest possible, as clarified in Corollaries \ref{cor:outmcg} and \ref{cor:commmcg}. This proves the two parts of Theorem \ref{thmintro:alg}.

Section~\ref{sec:examples} contains examples of large translation quotients satisfying Convention~\ref{conv:surface-inherited}. Most notably, in Lemma~\ref{lem:random_are_LT} we invoke results of Abbott, Berlyne, Ng, Rasmussen, and the author \cite{randomquot_HHG} to show that random quotients fit in our framework, thus proving Theorems~\ref{thmintro:comb_for_random} to \ref{thmintro:alg_for_rand} as special cases of Theorems~\ref{thmintro:comb} to \ref{thmintro:alg}.

Finally, in Section~\ref{sec:questions} we speculate on how one could try to extend combinatorial rigidity to prove that all \emph{injective} self-maps of $\CSN$ are induced by automorphisms. 

\subsection*{Acknowledgements}
I would like to thank Jason Behrstock, Mark Hagen, and Alessandro Sisto for sharing with me their previous attempts to find a different set of assumptions under which the conclusion of \cite[Theorem 5.7]{quasiflat} holds, and for suggesting many corrections to the first draft of this paper. I am especially grateful to my supervisor Alessandro Sisto for his constant support. I am also grateful to Piotr Przytycki for suggesting a way to shorten the proof of Theorem \ref{thm:ivanov_S5}. Finally, I thank Carolyn Abbott, Daniel Berlyne, Thomas Ng, and Alexander Rasmussen for fruitful discussions.

\section{Large translation quotients}\label{sec:setting}
Let $S$ be a surface of finite type, that is, a surface obtained from a closed, connected, oriented surface after removing a finite number of points, called punctures. When we want to emphasise the genus $g$ and the number of punctures $p$ we use the notation $S_{g,p}$, and we define the \emph{complexity} of the surface as the quantity $\zeta(S_{g,p})=3g+p-3$. Unless otherwise stated, by a \emph{curve} we mean an isotopy class of simple, essential, closed curves. The \emph{curve graph}  $\C  S$ is the graph whose vertices are curves on $S$, and adjacency corresponds to disjointness. Finally, let $\MCG$ be the extended mapping class group of $S$, where we allow orientation-reversing mapping classes. If $x\in\CS^{(0)}$ is a curve and $f\in\MCG$ is a mapping class, we denote the image of $x$ under $f$ simply by $fx$.

\begin{defn}[Translation length]
    The \emph{translation length} of a mapping class $f\in\MCG$ is defined as $\min_{x\in\CS}\dist_{\CS}(x,f x)$. Similarly, the \emph{minimum translation length} of a subgroup $N\le \MCG$ is defined as 
    $$\min_{f\in N-\{1\}}\min_{x\in\CS}\dist_{\CS}(x,f x),$$
    where $1\in\MCG$ is the identity.
\end{defn}

\begin{defn}[Large translation quotient]\label{defn:lrq}
    Let $N\unlhd \MCG$ be a normal subgroup. We will say that the quotient group $\MCG/N$ is a \emph{large translation quotient} if the minimum translation length of $N$ is at least $9$.
\end{defn}

\subsection{Isometric projections}\label{subsec:projections}
For the rest of the paper, let $\MCG/N$ be a large translation quotient. Since $N$ acts on $\CS$ by simplicial automorphisms, we can consider the quotient $\pi\colon \CS\to \CSN$. Given a subgraph $\ov  X\subseteq \CSN$, we say that a subgraph $X$ of $\CS$ is a \emph{lift} of $\ov  X$ if the projection map $\pi$ restricts to an isometry between $X$ and $\ov  X$. For later purposes, we gather here several properties of the quotient:

\begin{lem}[Isometric lifts and projections]\label{lem:lift}
    The following facts hold:
    \begin{itemize}
    \item $\C S/N$ is a simplicial graph.
    \item The projection $\pi\colon \CS \to \CSN$ is $1$-Lipschitz;
    \item For every combinatorial path $\ov  \Gamma\subset \CSN$ there exists a combinatorial path $\Gamma\subset \CS$ such that $\pi(\Gamma)=\ov \Gamma$, and moreover if $\ov  \Gamma$ is geodesic then so is $\Gamma$;
    \item  $\pi$ is a local isometry. More precisely, for every $x\in \CS^{(0)}$, $\pi$ restricts to an isometry between the closed ball of radius $2$ centred at $x$, which we denote by $\ov B(x,2)$, and the closed ball of radius $2$ centred at its projection $\ov  x$, which we denote by $\ov  B(\ov  x,2)$;
    \item Every subgraph $\ov  X$ of $\CSN$ which is contained in a closed ball of radius $2$ admits a unique $N$-orbit of lifts.
    \end{itemize}
\end{lem}

\begin{proof}
    By definition, edges in $\CSN$ correspond to $N$-orbits of edges in $\CS$. First notice that, by the assumption on the minimum translation length, two vertices in the same $N$-orbit cannot be adjacent, hence $\CSN$ has no loop edges. Moreover, if there were two edges between two vertices $\ov  x,\ov  y$ of $\CSN$, we could find two edges $\{x,y\}$ and $\{x',y'\}$ between representatives of $\ov  x$ and $\ov  y$, respectively. We could further assume that $y=y'$, up to replacing $\{x',y'\}$ with one of its $N$-translates. But then $x$ and $x'$ would be $N$-translate within distance at most $2$, contradicting the bound on the minimum translation length.
        
    The quotient map is $1$-Lipschitz since the action of $N$ on $\CS$ is simplicial. In order to find a combinatorial path $\Gamma\subset \CS$ which projects to a given path $\ov\Gamma\subset \CSN$ it suffices to lift one edge at a time, given a lift of its starting point, and again this can be done since the action is simplicial. Notice that, by construction, $\Gamma$ has the same length as $\ov\Gamma$, hence if $\ov \Gamma$ is a geodesic then so is $\Gamma$, as otherwise we could find a shorter path $\Gamma'$ between the endpoints of $\Gamma$ which would project to a shorter path between the endpoints of $\ov \Gamma$.
    
     
    Now, let $x\in\CS^{(0)}$ and let $\ov  x$ be its projection. Combining the previous points we have that $\pi(\ov B(x,2))=\ov  B(\ov  x,2)$, so we are left to show that, for every $y,z\in \ov B(x,2)$, we have that $\dist_{\CSN}(\ov  y, \ov  z)=\dist_{\CS}(y,z)$, where $\ov  y$ and $\ov  z$ are the respective projections. If by contradiction this was not the case, then we could pick any geodesic $\ov  \Gamma$ between $\ov  y$ and $\ov  z$. Notice that, since $y,z\in \ov B(x,2)$ and $\pi$ is $1$-Lipschitz, we have that $\dist_{\CSN}(\ov  y, \ov  z)\le 4$, thus $\ov\Gamma$ must have length at most $3$. Now lift $\ov\Gamma$ to a geodesic path $\Gamma$ starting at $y$ and ending at some $z'$ in the same orbit of $z$. But then $z$ and $z'$ would be within distance at most $7$, contradicting the assumption on the minimum translation length.

    For the last statement, every subgraph $\ov  X$ which is contained in a closed $2$-ball admits a lift, by applying the local inverse of $\pi$. Given any two lifts $X,X'$, let $x\in X$ and $x'\in X'$ be two points with the same projection, so that there exists an element $n\in N$ such that $x'=nx$. Now pick another pair $y\in X$ and $y'\in X'$ with the same projection. Since $y'$ and $ny$ are within distance at most $8$, by the assumption on the minimum translation length we must have that $y'=ny$. Hence $X'=nX$, that is, all lifts belong to the same $N$-orbit.
\end{proof}

\section{Combinatorial rigidity}\label{sec:combrig}
\noindent The $\MCG$-action on $\CS$ naturally induces a simplicial action of $\MCG/N$ on $\CSN$. The goal of this Section is to show that, for most non-sporadic surfaces, the associated map $\MCG/N\to \Aut(\CSN)$ is an isomorphism, thus proving Theorem \ref{thmintro:comb} from the Introduction. For the sake of simplicity, we first state the result for surfaces of sufficiently large complexity, postponing low-complexity cases to Subsection \ref{subsec:ivanov_sporadic}.

\begin{thm}[Combinatorial rigidity]\label{thm:ivanov}
    Let $S$ be a surface of complexity at least $3$ which is not $S_{2,0}$, and let $\MCG/N$ be a large translation quotient. Then the map $\MCG/N\to \Aut(\CSN)$, induced by the natural action, is an isomorphism. 
\end{thm}

\noindent For this Section only, it is convenient to see $\CS$ and $\CSN$ as simplicial \emph{complexes}, that is, we consider the flag simplicial complexes whose $1$-skeletons are the curve graph and its quotient, respectively. By a result of Harer \cite{Harer:CS_wedge}, $\CS$ is simply connected whenever $S$ has complexity at least $3$. Furthermore, the assumption on the large translation easily implies that $\CS$ is actually the universal cover of $\CSN$: 

\begin{lem}\label{lem:covering}
    The kernel $N$ acts properly discontinuously on $\CS$. As a consequence, the projection $\pi\colon \CS \to \CSN$ is a regular covering map, and the group of deck transformations is $N$.
\end{lem}

\begin{proof}
    For any point $p\in \CS$ let $x\in \CS^{(0)}$ be any vertex such that $p$ belongs to a simplex containing $x$, and let $U_p$ be the interior of the subcomplex spanned by $\ov B(x,2)$, which is an open neighbourhood of $p$. Then for every $n\in N-\{1\}$ we have that $U_p\cap n\cdot U_p=\emptyset$, as $\dist(x,nx)>4$. 
\end{proof}

\begin{proof}[Proof of Theorem \ref{thm:ivanov}]
    We first argue that the map $\MCG/N\to \Aut(\CSN)$ is surjective, i.e. that given an automorphism $\ov  \phi\in \Aut(\CSN)$ we can find an element $\ov  g\in \MCG/N$ inducing $\ov \phi$. Since a simplicial map between simplicial complexes is continuous, and since $\pi\colon \CS\to\CSN$ is the universal cover, there exists a continuous map $\phi\colon \CS\to \CS$ that makes the following diagram commute:
    $$\begin{tikzcd}
        \CS\ar{r}{\phi}\ar{d}{\pi}&\CS\ar{d}{\pi}\\
        \CSN\ar{r}{\ov \phi}&\CSN
    \end{tikzcd}$$
    It is easily seen that $\phi$ must map vertices to vertices and edges to edges, since $\pi$ is injective on every edge. Hence $\phi$ is a simplicial map, and actually an isomorphism as the same argument with $\ov \phi$ replaced by $\ov \phi^{-1}$ produces an inverse of $\phi$. In turn, by Ivanov's Theorem \cite{Ivanov:autC, Korkmaz} there is an element $g\in\MCG$ inducing $\phi$, and therefore its image $\ov g\in\MCG/N$ induces $\ov \phi$.

    
    Next we turn to injectivity: we pick an element $g\in \MCG$ which induces the identity on $\CSN$, and we claim that $g\in N$. To see this, notice that $\pi\circ g=g$, so $g$ acts as a deck transformation. Thus Lemma~\ref{lem:covering} yields an element $n\in N$ such that the composition $gn$ is the identity on $\CS$, and since $S$ has complexity at least $3$ and is not an $S_{2,0}$ we have that $g=n^{-1}$ (see e.g. the discussion in \cite[Section 3.4]{FarbMargalit}).
\end{proof}

\subsection{Some low-complexity cases}\label{subsec:ivanov_sporadic}
In this Subsection we establish an analogue of Theorem \ref{thm:ivanov} for some surfaces of low complexity. Firstly, if $S=S_{2,0}$, Ivanov's theorem states that the map $\MCG\to \Aut({\CS})$ is surjective, with kernel generated by the \emph{hyperelliptic involution} from Figure \ref{fig:hyp_involution} \cite{Ivanov:autC}. Hence arguing exactly as in the proof of Theorem \ref{thm:ivanov} we get:
\begin{prop}\label{thm:ivanov_S20}
    Let $S=S_{2,0}$, and let $\MCG/N$ be a large translation quotient. Then the map $\MCG/N\to \Aut(\CSN)$ is surjective and has finite kernel $\ov  K$, generated by the image of the hyperelliptic involution from Figure \ref{fig:hyp_involution}.
\end{prop}

\noindent Next, assume that $S$ is either $S_{0,4}$, $S_{1,0}$ or $S_{1,1}$. In this case the curve graph is defined differently: we connect two (isotopy classes of) curves if and only if they realise the minimal intersection number among all pairs of curves on the surface, which is $1$ for $S_{1,0}$ and $S_{1,1}$ and $2$ for $S_{0,4}$. The resulting simplicial complex is the Farey complex, which is a triangulation of the compactification of the hyperbolic plane and therefore is simply connected  (see e.g. \cite{Minsky_geom_approach} for a proof). Moreover, one can define a large translation quotient exactly as in Definition \ref{defn:lrq}, and prove the analogue of Lemma \ref{lem:covering} with the same arguments. Then again we can run the proof of Theorem \ref{thm:ivanov} using the version of Ivanov's theorem for surfaces of low complexity (see \cite{Korkmaz}), and we get the following:
\begin{prop}\label{thm:ivanov_sporadic}
    Let $S$ be either $S_{0,4}$, $S_{1,0}$, or $S_{1,1}$, and let $\MCG/N$ be a large translation quotient. Then the map $\MCG/N\to \Aut(\CSN)$ is surjective and has finite kernel $\ov  K$, generated by the images of the \emph{hyperelliptic involutions} in Figure \ref{fig:hyp_involution}.
\end{prop}

\begin{figure}[htp]
    \centering
    \includegraphics[width=0.8\textwidth]{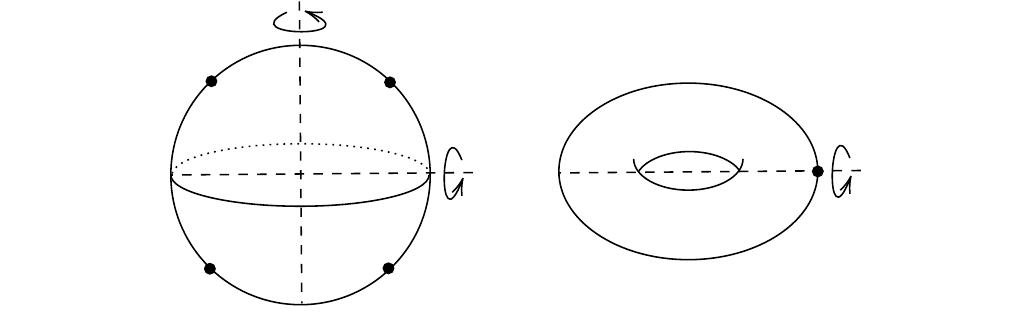}
    \includegraphics[width=0.8\textwidth]{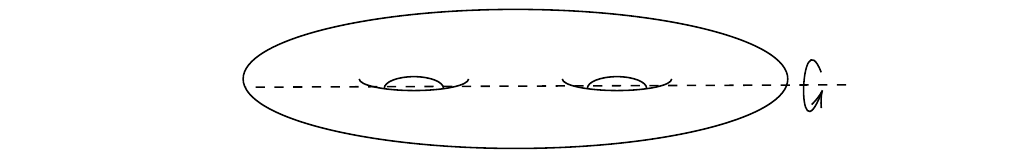}
    \caption{The dots represent the punctures. The hyperrelliptic involutions are the rotations by $\pi$ about the indicated axes, which fix all curves on the surfaces.}
    \label{fig:hyp_involution}
\end{figure}

\noindent Finally, leaving the case $S=S_{1,2}$ aside, which is more involved as the natural map $\mathcal{MCG}(S_{1,2})\to \Aut(\C S_{1,2})$ is not surjective \cite{Luo}, we are left to deal with the case $S=S_{0,5}$. What makes this surface different is that the curve complex $\C(S_{0,5})$ has dimension $1$ and is not a tree, and is therefore not simply connected. However, we can still run the proof of Theorem \ref{thm:ivanov} if we show that the polyhedral complex, obtained from $\C(S_{0,5})$ by gluing a cell to every $5$-cycle, is simply connected. The latter is unpublished work of Piotr Przytycki, whom I thank for explaining his proof to me.

First, by an \emph{arc} we mean the isotopy class, relative to the punctures, of a map $\gamma\colon [0,1]\to S_{0,5}$ mapping each endpoint of the interval to a puncture. For every surface $S$ with at least $2$ punctures, let $\mathcal{A}^2(S)$ be the simplicial complex whose vertices are all arcs $\gamma\colon [0,1]\to S_{0,5}$ whose endpoints are \emph{distinct} punctures, and where a collection of arcs span a simplex if their interiors are pairwise disjoint. 

\begin{prop}
    The complex $\mathcal{A}^2(S)$ is simply-connected.
\end{prop}

\noindent This fact is proven as \cite[Claim 3.17]{Schleimer}, but we spell out the details for clarity. A similar argument actually yields that $\mathcal{A}^2(S)$ is contractible, though we will not need this.

\begin{proof}
     Let $\gamma\colon \mathbb{S}^1\to \mathcal{A}^2(S)$ be a simplicial loop, and fix an arc $x_0\in\gamma$. If every other arc $y\in \gamma$ is disjoint from $x_0$, then we can connect $x_0$ to $y$, and this defines an extension of $\gamma$ to some triangulation of the disk $D^{2}$. 
    
    Thus, suppose that $\Sigma(\gamma)=\sum_{y\in\gamma} i(x_0,y)>0$, where $i(x_0,y)$ is the number of intersection points between two arcs in minimal position representing $x_0$ and $y$. Fix and endpoint $p_0$ of $x_0$, and let $y_0\in\gamma$ be the first arc one meets when travelling along $x_0$ starting from $p_0$. Let $t$ be the path along $x_0$ from $p_0$ to the first intersection with $y_0$, and let $z$ and $z'$ be the arcs obtained as the union of $t$ and one of the components of $y_0-t$. Since $y_0$ has distinct endpoints, one of its endpoints is not $p_0$, and therefore one of the new arcs, say $z$, has distinct endpoints. Notice moreover that $z$ is disjoint from $y_0$, and if $y_1\in \gamma$ was disjoint from $y_0$ then it is also disjoint from $z$, because we choose $y_0$ to have the closest intersection to $p_0$. Thus, we can replace $y_0$ with $z$ and obtain a loop $\gamma'$, which is homotopic to $\gamma$ but is such that $\Sigma(\gamma')<\Sigma(\gamma)$. Then we conclude by induction on $\Sigma(\gamma)$.
\end{proof}

\begin{defn}\label{defn:pentagon}
    Let $S=S_{0,5}$. By a \emph{pentagon} we mean a $5$-cycle $P\subset \CS$. Let $\CS_P$ be the polyhedral complex obtained from $\CS$ by gluing a $2$-cell to every pentagon.
\end{defn}

\begin{prop}[Przytycki]
    The complex $\CS_P$ is simply connected.
\end{prop}

\begin{proof}
    First, notice that $\CS$ can be seen as a non-complete subgraph of $\mathcal{A}^2(S)$, by mapping every curve $y$ to the unique arc $a(y)$ in the twice-punctured disk that $y$ cuts out on $S$. The map $a\colon \CS\to\mathcal{A}^2(S)$ is bijective at the level of vertices, but not at the level of edges, since in $\mathcal{A}^2(S)$ we allow two arcs to have the same endpoints. Notice moreover that, if the union of five arcs is a pentagon on $S$ (meaning that there is a cyclic ordering such that each arc intersects only the previous and the following, and only at a single endpoint), then the preimages under $a$ of these arcs form a pentagon, in the sense of Definition \ref{defn:pentagon}.
    
    Now let $\gamma\colon \mathbb{S}^1\to \CS_P$ be a simplicial loop, whose image lies in the $1$-skeleton $\CS_P^{(1)}=\CS$. The image $a(\gamma)\subset \mathcal{A}^2(S)$ can be filled with triangles, since $\mathcal{A}^2(S)$ is simply connected. Let $x_0,x_1,x_2\in \mathcal{A}^2(S)$ be the vertices of such a triangle $T$. If $x_i$ and $x_j$ have disjoint endpoints, then the corresponding curves $a^{-1}(x_i)$ and $a^{-1}(x_j)$ are disjoint, so that they are connected by an edge in $\CS$. Otherwise, there exists a unique arc $\varepsilon_{ij}\in\mathcal{A}^2(S)$ such that $a^{-1}(\varepsilon_{ij})$ is disjoint from both $a^{-1}(x_i)$ and $a^{-1}(x_j)$, as in Figure \ref{fig:disjoint_arcs} (it is the only arc with distinct endpoints in the complement of $x_i\cup x_j$). 

    \begin{figure}[htp]
        \centering
        \includegraphics[width=\textwidth]{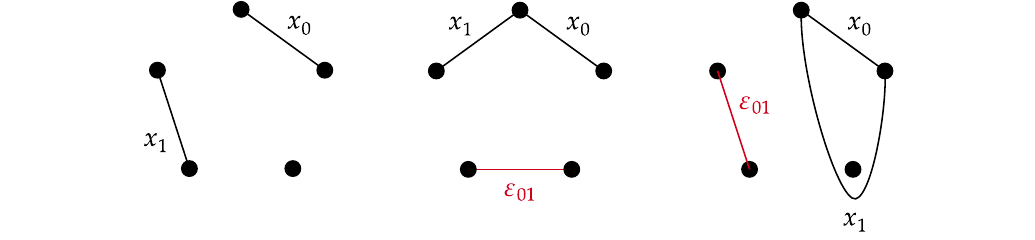}
        \caption{The possible configurations for two arcs $x_0$ and $x_1$, and the corresponding arc $\varepsilon_{01}$.}
        \label{fig:disjoint_arcs}
    \end{figure}

    By construction, the curves
    $a^{-1}\left(\{x_i\}_{i=0,1,2} \cup \{\varepsilon_{ij}\}_{0\le i<j\le 2}\right)$
    form a loop $\delta$ in $\CS$. Therefore we are left to prove that $\delta$ can be filled with finitely many pentagons.

    If there is a pair of arcs, say, $x_0$, $x_1$, which have distinct endpoints, then $x_2$ must share an endpoint with each of them. This shows that $\delta$ has precisely $5$ vertices, and therefore is filled by a pentagon in $\CS_P$. Otherwise, suppose that all pairs $x_i$, $x_j$ share an endpoint, that is, $\delta$ has six vertices. There are five possible cases, as in Figure \ref{fig:three_arcs}.

    \begin{figure}[htp]
        \centering
        \includegraphics[width=\textwidth]{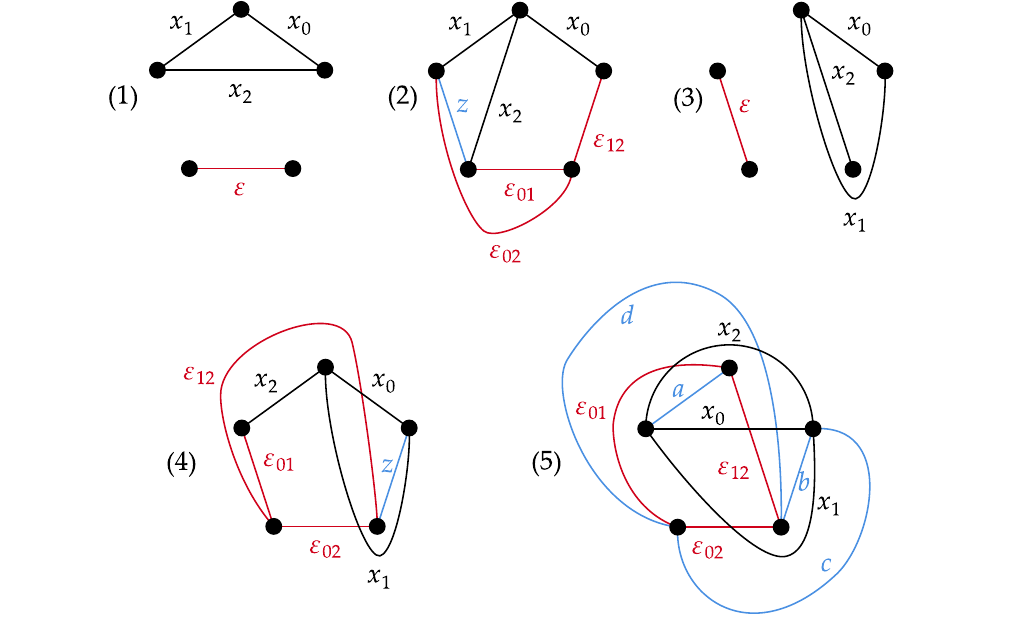}
        \caption{The possible configurations for three arcs $x_0,x_1,x_2$, and the corresponding arcs $\varepsilon_{ij}$.}
        \label{fig:three_arcs}
    \end{figure}

    Let us analyse the above cases separately.
    \begin{enumerate}[(1)]
    \item\label{a}  The $\varepsilon$-arcs coincide, thus $\delta=a^{-1}\left(\{x_i\}_{i=0,1,2} \cup \{\varepsilon\}\right)$ is a tripod.
    \item\label{b} Add the arc $z$ in the Figure. Now $\{x_0,x_1,z,\varepsilon_{01},\varepsilon_{12}\}$ form a pentagon, as well as $\{x_0,x_2,z,\varepsilon_{02},\varepsilon_{12}\}$. Therefore we can fill $\delta$ with the preimage under $a$ of these arcs, which is a union of two pentagons joined along two edges.
    \item The $\varepsilon$-arcs coincide, so we are in the same situation as case \ref{a}.
    \item Similarly to case \ref{b}, adding $z$ creates two pentagons.
    \item This time we need four auxiliary curves, in order to form four pentagons $P_1=\{x_0, a,\varepsilon_{01},\varepsilon_{12},b\}$, $P_2=\{x_2, a,\varepsilon_{12},\varepsilon_{02},c\}$, $P_3=\{x_1, a,\varepsilon_{12},d,c\}$ and $P_4=\{x_1, a,\varepsilon_{01},d,b\}$.\qedhere
    \end{enumerate}
\end{proof}

\begin{prop}\label{thm:ivanov_S5}
Let $S=S_{0,5}$, and let $\MCG/N$ be a large translation quotient. Then the map $\MCG/N\to \Aut(\CSN)$ is an isomorphism.
\end{prop}

\begin{proof}[Proof of Theorem \ref{thm:ivanov_S5}]
    Notice first that, as a consequence of Lemma \ref{lem:lift}, pentagons in $\CSN$ can be lifted, and pentagons in $\CS$ can be projected isometrically. Thus the quotient $\CS_P/N$ is again a polygonal complex. Moreover, any automorphism $\ov\phi\colon \CSN\to \CSN$ maps pentagons to pentagons, and therefore extends to an automorphism of $\CS_P/N$. We can now run the proof of Theorem \ref{thm:ivanov}, this time using the covering  $\CS_P\to \CS_P/N$, to show that the map $\MCG/N\to \Aut(\CSN)$ is an isomorphism (again, injectivity follows from \cite[Section 3.4]{FarbMargalit}). 
\end{proof}

\section{Quasi-isometries of hierarchically hyperbolic spaces}\label{sec:qusiiso_of_HHG}
\noindent For this Section only, we work in the general framework of \emph{hierarchically hyperbolic spaces and groups} (resp. HHS and HHG), first introduced by Behrstock, Hagen, and Sisto in \cite{HHSI}. We will show that, if a HHS $X$ satisfies some additional properties (see Assumptions \eqref{ass1} - \eqref{ass6} and Remark \ref{rem:asymph}), then any self-quasi-isometry of $X$ induces an automorphism of a certain graph, which encodes the intersection patterns of certain top-dimensional quasiflats. Once again, I am grateful to Jason Behrstock, Mark Hagen, and Alessandro Sisto for their contribution, in particular to the proof of Lemma \ref{lem:orth_go_to_orth_or_//}.

\subsection{Background on hierarchical hyperbolicity}\label{sec:hhs}
We first recall the intuition behind some notions from the world of hierarchical hyperbolicity, always keeping the motivating example of mapping class groups at hand. We refer to \cite{HHSII} for further details. 

\subsubsection{HHS and HHG} A \emph{hierarchically hyperbolic space} (HHS for short) is a metric space $X$ that comes with certain additional data, most importantly a family of uniformly hyperbolic spaces $\{\C Y\}_{Y\in \mathfrak S}$, called \emph{coordinate spaces}, and uniformly coarsely Lipschitz maps $\pi_Y:X\to \C Y$, which should be thought of as coordinate projections. For mapping class groups, these are curve graphs of subsurfaces and maps coming from subsurface projections. Moreover, the \emph{domain set} $\mathfrak S$, that is, the set indexing the family of coordinate spaces, has a partial ordering $\nest$, called \emph{nesting}, with a unique maximal element $S$, and a symmetric relation $\orth$, called \emph{orthogonality}. For mapping class groups, these are containment and disjointness of subsurfaces (up to isotopy). Nesting and orthogonality are mutually exclusive, meaning that if $U\nest V$ then $U\not\bot V$. When two domains $U,V$ of $\mathfrak S$ are not $\nest$- nor $\orth$-related, one says that they are \emph{transverse}, and one writes $U\transverse V$. Finally, whenever $U\transverse V$ there is a uniformly bounded subset $\rho^U_V\subseteq \C V$, which for mapping class group is again defined using subsurface projection.

An \emph{action} of a finitely generated group $G$ on a hierarchically hyperbolic space $(X,\frakS)$ is the data of:
\begin{itemize}
\item an action $G\circlearrowleft X$ by isometries;
\item an action $G\circlearrowleft \frakS$, preserving nesting and orthogonality;
\item for every $g\in G$ and every $Y\in \frakS$, an isometry $g_{Y}\colon \C Y\to \C g(Y)$.
\end{itemize}

Moreover, one requires that the two actions are compatible, meaning that for every $g\in G$ and every transverse domains $U,V\in \frakS$ the following diagrams commute:

$$\begin{tikzcd}
X\ar{r}{g}\ar{d}{\pi_U}&X\ar{d}{\pi_U}\\
\C U\ar{r}{g_U}&\C g(U)
\end{tikzcd}
\begin{tikzcd}
\C U\ar{r}{g_U}\ar{d}{\rho^U_V}&\C g(U)\ar{d}{\rho^{g(U)}_{g(V)}}\\
\C V\ar{r}{g_V}&\C g(V)
\end{tikzcd}$$

\noindent We often slightly abuse notation and drop the subscript for the isometry $g_U$. If the action on $X$ is metrically proper and cobounded, and the action on $\frakS$ is cofinite, then $G$ is a \emph{hierarchically hyperbolic group} (HHG for short), and any quasi-isometry between $G$ and $X$ given by the Milnor-\v{S}varc lemma endows $G$ with the HHS structure of $X$. 

\begin{rem}\label{rem:asymph}
    Hierarchically hyperbolic groups are examples of \emph{asymphoric} HHS, in the sense of \cite[Definition 1.14]{quasiflat}. This property is a requirement of many lemmas from \cite{quasiflat}, but we omit its definition as we shall never explicitly use it.
\end{rem}

\subsubsection{Product, flats, orthants}
The idea of orthogonality is that it corresponds to products, in the following sense. Given any $U\in \mathfrak S$, there is a corresponding space $F_U$ associated to it, which is quasi-isometrically embedded in $X$ ($F_U$ is a HHS itself with domain set $\mathfrak S_U=\{Y\in \mathfrak S\mid Y\nest U\}$, and in mapping class groups $F_U$ roughly corresponds to the mapping class groups of $U$). Given a maximal set $\{U_i\}$ of pairwise orthogonal elements of $\mathfrak S$, there is a corresponding \emph{standard product region} $P_{\{U_i\}}$ which is quasi-isometric to the product of the $F_{U_i}$ (think of a Dehn twist flat as, coarsely, a product of annular curve graphs).

Moreover, inside $X$ and the $F_U$s there are special uniform quasi-geodesics, called \emph{hierarchy paths}, which are those that project monotonically (with uniform constants) to all $\C Y$. Similarly, there are hierarchy rays and hierarchy lines, which are quasigeodesic rays and lines, respectively.  Given a pairwise orthogonal collection $\{U_i\}$, if for each $U_i$ we either choose a hierarchy line or a point, the product region $P_{\{U_i\}}$ contains a product of the given lines. This is what we will refer to as a \emph{standard $k$-flat}, where $k$ is the number of lines; we can analogously define \emph{standard $k$-orthants} as products of hierarchy rays. The \emph{support} of a standard $k$-flat (resp., orthant) is the set of $U_i$ for which a line (resp. ray) has been assigned.

\subsubsection{Complete support sets and hinges}
A \emph{complete support set} is a collection $\{U_i\}_{i=1}^\nu\subseteq \mathfrak S$ of pairwise orthogonal domains with all $\C U_i$ unbounded, and with maximal cardinality $\nu$ among sets with these properties. The quantity $\nu$ is called the \emph{rank} of $(X,\frakS)$, and if for every $U\in\frakS$ the Gromov boundary of $\C U$ is non-empty, then $\nu$ is the maximal dimension of standard $k$-orthants, as we now clarify.

A \emph{hinge} is a pair $\sigma =(U,p)$ where $U$ belongs to a complete support set and $p\in\partial \C U$. We say that $U$ is the \emph{support} of $\sigma$. As in \cite[Definition 5.3]{quasiflat}, one can associate to a hinge $\sigma$ a standard $1$-orthant, called its \emph{hinge ray} and denoted $\mathfrak h_\sigma$. This is a hierarchy ray whose projection to $\C U$ is a quasigeodesic ray asymptotic to $p$, and whose projections to all other coordinate spaces are uniformly bounded. Furthermore, \cite[Remark 5.4]{quasiflat} states that if $\sigma\neq\sigma'$ are two different hinges then $\dist_{Haus}(\mathfrak h_\sigma,\mathfrak h_{\sigma'})=\infty$.

Now, given a complete support set $\{U_i\}_{i=1}^\nu$ and a choice of a point $p_i\in \partial\C U_i$ for every $i$, the product of the $\mathfrak h_{(U_i,p_i)}$ is a standard $\nu$-orthant. Similarly, if we are given a pair of distinct points $p_i^{\pm}\in \partial\C U_i$ for every $i$, we can construct a standard $\nu$-flat, denoted by $\mathfrak F_{\{(U_i,p_i^{\pm})\}}$. We will refer to standard $\nu$-flats (resp. standard $\nu$-orthants) simply as \emph{standard flats} (resp. \emph{standard orthants}).

The next Lemma follows from carefully inspecting the proofs of \cite[Lemmas 4.11 and 4.12]{quasiflat}, and describes the coarse intersection of standard flats and orthants. Recall that, given subsets $A,B$ of a metric space $X$, the \emph{coarse intersection} $A\Tilde{\cap} B$, if well-defined, is a subspace of $X$ within bounded Hausdorff distance of all intersections between the $R$-neighbourhoods of $A$ and $B$, for $R$ sufficiently large. 

\begin{lem}[Coarse intersection of standard orthants and flats]\label{lem:coarse_int_orthants}
    Let $(X,\frakS)$ be an asymphoric HHS of rank $\nu$.
    \begin{itemize}
        \item Let $\mathcal O, \mathcal O'$ be standard orthants in $X$ with supports $\{U_i\}_{i=1}^\nu$ and $\{V_i\}_{i=1}^\nu$. Then $\mathcal O \Tilde{\cap} \mathcal O'$ is well-defined, and coarsely coincides with a standard $k$–orthant whose support is contained in $\{U_i\}_{i=1}^\nu\cap\{V_i\}_{i=1}^\nu$.
        \item Let $\mathcal F, \mathcal F'$ be standard orthants in $X$ with supports $\{U_i\}_{i=1}^\nu$ and $\{V_i\}_{i=1}^\nu$. Then $\mathcal F \Tilde{\cap} \mathcal F'$ is well-defined, and coarsely coincides with a product $\prod_{T} \ell_T$, where $T$ varies in $\{U_i\}_{i=1}^\nu\cap\{V_i\}_{i=1}^\nu$ and every $\ell_T$ is either a point, a standard $1$-orthant or a standard $1$-flat supported on $T$.
    \end{itemize}
\end{lem}

\subsection{Quasi-isometries induce maps on hinges}\label{subsec:hinge}
Let $(X,\frakS)$ be an asymphoric HHS of rank $\nu$, and let $\bf{Hinge}(\frakS)$ be its set of hinges. We now prove that, under certain assumptions that we present as we go, every quasi-isometry of $X$ induces a map of a graph encoding the intersection patterns of standard flats. 

\begin{ass}\label{ass1}
    For every $U$ which belongs to a complete support set, the Gromov boundary $\partial \C  U$ of its coordinate space consists of at least $2$ points.
\end{ass}

\begin{defn}[Good domain]\label{def:good}
    A domain $U\in \frakS$ is \emph{good} if there exist two complete support sets whose intersection is $\{U\}$. Let $\frakS_G$ be the set of good domains, and let $\mathbf{Hinge}_G(\frakS)$ be the set of hinges supported on good domains.
\end{defn}

\begin{ass}\label{ass2}
    For every $U,V\in \frakS_G$ which are orthogonal, there exist two support sets whose intersection is $\{U,V\}$.
\end{ass}

\begin{ass}\label{ass3}
    For every $U\in \frakS$ which belongs to a complete support set, there exists a (possibly empty) collection of good domains $\{V_1,\ldots,V_n\}$ and two complete support sets whose intersection is $\{U,V_1,\ldots,V_n\}$. 
\end{ass}

\begin{rem}\label{rem:f_hin}
    Let $f\colon X \to X$ be a quasi-isometry. With the same arguments as in the first half of the proof of \cite[Theorem 5.7]{quasiflat}, one can define a map $f_{Hin}\colon \mathbf{Hinge}_G(\frakS)\to \mathbf{Hinge}(\frakS)$ by mapping $\sigma\in\mathbf{Hinge}_G(\frakS)$ to the unique hinge $\sigma'\in \mathbf{Hinge}(\frakS)$ such that $\dist_{\text{Haus}}(f(\mathfrak h_\sigma), \mathfrak h_{\sigma'})<+\infty $. Moreover, by inspection of the construction one sees that $\sigma'$ is itself good, since $\mathfrak h_{\sigma'}$ arises as the coarse intersection of two standard orthants (whose supports must therefore coincide only on the support of $\sigma'$ by Lemma \ref{lem:coarse_int_orthants}). Hence $f_{Hin}$ maps $\mathbf{Hinge}_G(\frakS)$ to itself; moreover, if $g$ is any quasi-inverse for $f$, then $f_{Hin}$ and $g_{Hin}$ are inverses, showing that $f_{Hin}$ is a bijection.
\end{rem}

\begin{defn}
    We say that two hinges are \emph{co-supported} if they are supported on the same domain, and  are \emph{orthogonal} if they are supported on orthogonal domains.
\end{defn}

\begin{lem}[Behrstock, Hagen, Sisto]\label{lem:orth_go_to_orth_or_//} 
    Under Assumptions \eqref{ass1} - \eqref{ass3}, if two good hinges $\sigma, \sigma'\in \mathbf{Hinge}_G(\frakS)$ are orthogonal then $f_{Hin}(\sigma)$ and $f_{Hin}(\sigma')$ are either orthogonal or co-supported.
\end{lem}

\begin{proof}
    Let $\sigma=(U,p)$ and $\sigma'=(U',p')$. By Assumption \eqref{ass2} there exist two  standard flats $\mathcal F_1,\mathcal F_2$ whose coarse intersection is a standard 2-flat containing $\mathfrak h_{\sigma}$ and $\mathfrak h_{\sigma'}$. Let $\mathcal O$ be the $2$-orthant spanned by $\mathfrak h_{\sigma}$ and $\mathfrak h_{\sigma'}$ in this standard 2-flat.
     
    As a consequence of the Quasiflats Theorem \cite[Theorem A]{quasiflat} and of Lemma \ref{lem:coarse_int_orthants}, $f(\mathcal F_1)\tilde{\cap}f(\mathcal F_2)$ is a union of standard 2-orthants, whose “boundary” 1-orthants, which we call \emph{coordinate rays}, can be ordered cyclically. Moreover, $f(\mathfrak h_{\sigma})$  must arise as some coordinate ray, since it coarsely coincides with some hinge $\mathfrak h_{f_{Hin}(\sigma)}$, and if a hinge is contained in a standard 2-orthant then it must be one of its boundary 1-orthants by \cite[Lemma 4.11]{quasiflat}. The same is true for $f(\mathfrak h_{\sigma'})$.
     
    Now, if $f(\mathfrak h_{\sigma})$ and $f(\mathfrak h_{\sigma'})$ are adjacent in the cyclic ordering, then they belong to a common 2-orthant, and therefore they are orthogonal. Thus suppose that there is a coordinate ray $r\in f(\mathcal O)$ between $f(\mathfrak h_{\sigma})$ and $f(\mathfrak h_{\sigma'})$, and let $V\in \frakS$ be its support. 
     
    If $V$ is good then we can proceed as in the proof of \cite[Theorem 5.7]{quasiflat} to get a contradiction. Indeed, in this case $f^{-1}(r)$ would coincide with some hinge ray $\mathfrak h_{\sigma''}$, which belongs to $\mathcal O$ but lies at infinite Hausdorff distance from both $\mathfrak h_{\sigma}$ and $\mathfrak h_{\sigma'}$, and this would contradict \cite[Lemma 4.11]{quasiflat}.
     
    Then we can assume that no ray between $f(\mathfrak h_{\sigma})$ and $f(\mathfrak h_{\sigma'})$ is supported on a good domain, and in particular that $V$ is not good. By Assumption \eqref{ass3} there exists a collection of good domains $\{V_1,\ldots,V_n\}$ and two complete support sets whose intersection is $\{V,V_1,\ldots,V_n\}$. Hence, there exist two standard flats $\mathcal G_1,\mathcal G_2$ whose intersection is a flat supported on $\{V,V_1,\ldots,V_n\}$ and containing $r$. Now consider $Y=\mathcal G_1\Tilde{\cap} \mathcal G_2\Tilde{\cap} f(\mathcal F_1)\Tilde{\cap} f(\mathcal F_2)$. This coarse intersection is well-defined by Lemma \ref{lem:coarse_int_orthants}, since both $f(\mathcal F_1)$ and $f(\mathcal F_2)$ are finite unions of standard orthants by the Quasiflat Theorem. Moreover, let $Z=f^{-1}(Y)$, which is a union of standard 1- and 2-orthants supported on $\{U,U'\}$ (again, this follows from combining the Quasiflat Theorem, applied to $f^{-1}(\mathcal G_1)$ and $f^{-1}(\mathcal G_2)$, and Lemma \ref{lem:coarse_int_orthants}). 
      
    Now, $Z$ contains the ray $f^{-1}(r)$, which lies inside $\mathcal O$ but cannot coarsely coincide with $\mathfrak h_{\sigma}$ nor with $\mathfrak h_{\sigma'}$, since their image lie at infinite Hausdorff distance. Therefore, $Z$ must contain the whole 2-orthant $\mathcal O$. In turn, this means that $Y$ must contain all 2-orthants of $f(\mathcal F_1)\tilde{\cap}f(\mathcal F_2)$ which lie between $f(\mathfrak h_{\sigma})$ and $f(\mathfrak h_{\sigma'})$ in the cyclic ordering, and moreover the support of every such 2-orthant must be contained in $\{V,V_1,\ldots,V_n\}$. Therefore, since $f(\mathfrak h_{\sigma})$ and $f(\mathfrak h_{\sigma'})$ are boundary 1-orthants of some 2-orthants of $Y$, their supports must belong to $\{V,V_1,\ldots,V_n\}$, and therefore they are either orthogonal or co-supported.
\end{proof}

\noindent For the next results we need the following additional assumptions:

\begin{ass}\label{ass4}
    If $U\in\frakS_G$ is a good domain, then $|\partial \C  U|=2$.
\end{ass}

\begin{ass}\label{ass5}
    For every two good domains $U\neq V\in\frakS_G$ there exists a good domain $W\in\frakS_G$ which is orthogonal to $U$ but not to $V$. In other words, each good domain is determined by the set of good domains which are orthogonal to it.
\end{ass}

\begin{lem}\label{lem://_go_to_orth_or_//}
    Under Assumptions \eqref{ass1} - \eqref{ass5}, if $\sigma, \sigma'\in \mathbf{Hinge}_G(\frakS)$ are co-supported then $f_{Hin}(\sigma)$ and $f_{Hin}(\sigma')$ are either orthogonal or co-supported.
\end{lem}

\begin{proof}
    Let $\sigma, \sigma'\in \mathbf{Hinge}_G(\frakS)$ be two co-supported good hinges. For every good hinge $\theta$ which is orthogonal to $f_{Hin}(\sigma)$ we can consider its preimage $(f_{Hin})^{-1}(\theta)$. By Lemma \ref{lem:orth_go_to_orth_or_//}, which also applies to $(f_{Hin})^{-1}$ (again, because it coincides with the map $g_{Hin}$ for any quasi-inverse $g$ for $f$), we have that $(f_{Hin})^{-1}(\theta)$ and $\sigma$ are either orthogonal or co-supported. Moreover, by Assumption \eqref{ass4} the unique other hinge which has the same support as $\sigma$ is $\sigma'$, thus $(f_{Hin})^{-1}(\theta)$ is orthogonal to $\sigma$ and therefore also to $\sigma'$. Then again Lemma \ref{lem:orth_go_to_orth_or_//} tells us that $\theta$ and $f_{Hin}(\sigma')$ are either co-supported or orthogonal. 
    
    We can repeat the argument for every good hinge $\theta$ which is orthogonal to $f_{Hin}(\sigma)$. This yields that either $f_{Hin}(\sigma')$ has the same support of some $\theta$, and therefore is orthogonal to $f_{Hin}(\sigma)$; or $f_{Hin}(\sigma)$ and $f_{Hin}(\sigma')$ are orthogonal to the same hinges, and therefore they are co-supported by Assumption \eqref{ass5}.
\end{proof}

\noindent For the final result, we also need good domains to satisfy a strengthening of Definition \ref{def:good}, which roughly says that good hinges arise as intersection of “good” standard orthants.

\begin{defn}[Very good domain]
    A domain $U\in \frakS$ is \emph{very good} if there exist two complete support sets, \emph{made of good domains}, whose intersection is $\{U\}$.
\end{defn}
\begin{ass}\label{ass6}
    Good domains are very good.
\end{ass}
\begin{prop}\label{lem:automorphism_of_frakS_G}
    Under Assumptions \eqref{ass1} - \eqref{ass6}, $f_{Hin}$ induces an automorphism $f_G$ of the graph $(\frakS_G,\orth)$, whose vertex set is $\frakS_G$ and where adjacency corresponds to orthogonality.
\end{prop}

\begin{proof}
    We regard $\mathbf{Hinge}_G(\frakS)$ as a simplicial graph, by saying that two hinges are adjacent if and only if they are orthogonal or co-supported. Combining Remark \ref{rem:f_hin}, Lemma \ref{lem:orth_go_to_orth_or_//}, and Lemma \ref{lem://_go_to_orth_or_//}, we obtain that $f$ induces a simplicial automorphism of $\mathbf{Hinge}_G(\frakS)$, which we still denote by $f_{Hin}$ with a slight abuse of notation.
     
    Notice that any maximal clique of $\mathbf{Hinge}_G(\frakS)$ is the (infinite) subgraph spanned by all hinges supported on some complete support set, made of good domains. Moreover, if two hinges are co-supported then they belong to the same maximal cliques; conversely, if $\sigma$ and $\sigma'$ are supported on $U\neq V$, respectively, then in view of Assumption \eqref{ass6} we can find a complete support set, made of good domains, which contains $U$ but not $V$. This proves that being co-supported is equivalent to belonging to the same maximal cliques, and is therefore a purely combinatorial property, which must be preserved by the automorphism $f_{Hin}$. The same argument, applied to $(f_{Hin})^{-1}$, yields that two good hinges $\sigma, \sigma'\in \mathbf{Hinge}_G(\frakS)$ are co-supported if and only if $f_{Hin}(\sigma)$ and $f_{Hin}(\sigma')$ are co-supported. In turn, by how we defined adjacency in  $\mathbf{Hinge}_G(\frakS)$, we get that two good hinges $\sigma, \sigma'\in \mathbf{Hinge}_G(\frakS)$ are orthogonal if and only if $f_{Hin}(\sigma)$ and $f_{Hin}(\sigma')$ are orthogonal.
    

    Then $f_{Hin}$ induces an automorphism $f_G$ of $(\frakS_G, \orth)$ as follows. For every $U\in\frakS$ choose a hinge $\sigma$ supported on $U$, and set $f_G(U)$ as the support of $f_{Hin}(\sigma)$. This map is well-defined since $f_{Hin}$ maps co-supported hinges to co-supported hinges, and it preserves orthogonality since $f_{Hin}$ does.
\end{proof}

\noindent The last Lemma of this Subsection gives a quantitative bound on the Hausdorff distance between the image of a maximal flat, supported on good domains, and the maximal flat supported on the image of the supports. In the case of $\MCG/N$, this will tell us that the quasi-isometry permutes “Dehn twist flats”, up to uniformly bounded distance. 

\begin{lem}[Flats go to flats]\label{lem:flats_to_flats}
    Let $(X,\frakS)$ be an asymphoric HHS of rank $\nu$, satisfying Assumptions \eqref{ass1} - \eqref{ass6}. For every $L\ge 0$ there exists $C\ge0$ such that the following holds. Let $f\colon X\to X$ be a $(L,L)$-quasi-isometry. Let $\{U_i\}_{i=1}^\nu$ be a complete support set, made of good domains, and let $\mathcal F_{\{U_i\}}$ be the standard flat it supports. Similarly, let $\mathcal F_{\{f_G(U_i)\}}$ the standard flat supported on $\{f_G(U_i)\}_{i=1}^\nu$. Then $\dist_{Haus}\left(f\left(\mathcal F_{\{U_i\}}\right), \mathcal F_{\{f_G(U_i)\}}\right)\le C$.
\end{lem}

\begin{proof}
    One can argue as in \cite[Lemma 5.9]{quasiflat}, whose proof only relies on asymphoricity of $X$ and the fact that $f_{Hin}$ preserves beign co-supported.
\end{proof}

\section{Quasi-isometric rigidity}\label{sec:qirigid}
\noindent Going back to mapping class groups, this Section is devoted to the proof of Theorem \ref{thmintro:qi} from the Introduction, that is quasi-isometric rigidity of those large translation quotients admitting a particular HHG structure (see Theorem \ref{thm:qirigid} for the exact statement). The idea is that, in this setting, the graph $(\frakS_G,\orth)$ corresponds to $\CSN$. Hence Lemma \ref{lem:automorphism_of_frakS_G} tells us that every quasi-isometry $f$ induces an automorphism of $\CSN$, which by the combinatorial rigidity results from Section \ref{sec:combrig} is induced by some element $\ov g\in\MCG/N$. With a little more effort, we will then show that $f$ coarsely coincides with the left multiplication by $\ov g$. 

\subsection{Surface-inherited HHG structures}
For the rest of the paper we shall only consider large translation quotient satisfying the following convention:
\begin{conv}[Surface-inherited HHG structure]\label{conv:surface-inherited} Let $S$ be a surface of complexity at least $2$ which is not a torus with two punctures. Let $\MCG/N$ be a large translation quotient, endowed with a HHG structure with the following properties:
\begin{itemize}
    \item The domain set $\ov \frakS=\frakS/N$ is the collection of $N$-orbits of essential subsurfaces.
   
    \item Two domains $\ov  U,\ov  V\in\ov \frakS$ are nested (resp. orthogonal) if and only if they admit representatives $U,V\in\frakS$ which are nested (resp. orthogonal).
    \item The top-level coordinate space $\C\ov S$ is $\CSN$. In particular, $\CSN$ is hyperbolic.
   
    \item Whenever $\ov  U\in\ov \frakS-\{S\}$, the associated coordinate space is $\left(\bigcup_{U\in \ov U} \C U\right)/N$, and the quotient projection $\bigcup_{U\in \ov U} \C U\to \C\ov U$ restricts to an isometry $\C U\to\C\ov U$, for any representative $U\in \ov U$. Here $\C U$ is the usual curve graph of $U$, unless $U$ has complexity $1$ and $\C U$ is isomorphic to the Farey complex, or $U$ is an annulus and $\C U$ is its annular curve graph (as defined in \cite[Section 2.4]{MM_2}).
\end{itemize}
\end{conv}
\noindent The assumption on the topological type of $S$, together with the combinatorial rigidity results from Section \ref{sec:combrig}, will ensure that every automorphism of $\CSN$ is induced by a mapping class. Moreover, we require the complexity to be at least $2$, so that in $\MCG/N$ we can find “Dehn Twist flats”, that is, standard flats supported on the orbits of pairwise disjoint annuli, of dimension at least $2$.

We devote the rest of the Subsection to showing that, if a large translation quotient satisfies Convention \ref{conv:surface-inherited}, then it fits the framework of Section \ref{subsec:hinge}. We first point out a few easy consequences of the Convention.
\begin{lem}\label{lem:css_lift}
    Every finite collection $\ov  U_1,\ldots, \ov  U_k\in \ov \frakS$ of pairwise orthogonal elements admits representatives $U_1,\ldots, U_k\in \frakS$ which are pairwise orthogonal.
\end{lem}

\begin{proof}
    The proof is by induction on $k$, the base case $k=2$ being true by Convention \ref{conv:surface-inherited}. Thus let $k>2$ and let $U_1,U_2$ be disjoint representatives for $\ov  U_1,\ov  U_2$. Moreover, let $U_3\in \ov  U_3$ be such that $U_3\orth U_1$ (such a representative exists since we know that $\ov  U_1,\ov  U_3$ admit disjoint representatives, and up to the $N$-action we can assume that the representative for $\ov  U_1$ is $U_1$). Similarly, there exists $n\in N$ such that $nU_3\orth U_2$. 

    Now fix a curve $x\in U_1$. First notice that every curve $y$ lying on $U_3$ is disjoint from $x$. Moreover, $n y$ is disjoint from all curves on $U_2$, which are in turn disjoint from $x$. Hence $\dist_{\CS}(y,n y)\le 3$, and by the assumption on the minimum translation length we must have that $n=1$. Thus $U_3$ is disjoint from both $U_1$ and $U_2$. If we repeat the procedure we can find representatives $U_i\in \ov  U_i$ for every $i=3,\ldots, k$ which are disjoint from both $U_1$ and $ U_2$. We can now replace $U_1$ and $U_2$ with the subsurface $U_1\sqcup U_2$, so that the conclusion follows by induction.
\end{proof}

\begin{lem}\label{lem:CSN_unbounded}
    $\CSN$ is unbounded.
\end{lem}

\begin{proof}
    Let $U,V$ be two annuli whose core curves $u,v$ are at distance $3$ in $\CS$. Notice that, by the assumption on the minimum translation length, for every $n\in N-\{1\}$, we have that $\dist_{CS}(u,nv)\ge 6$. This proves that any two translates of $U$ and $V$ fill the surface, and in particular $\ov U \transverse \ov V$. Furthermore, since both $U$ and $V$ are not the whole surface, $\C \ov U\cong \C U$ is unbounded, and similarly for $\C \ov V$.

    Now, by \cite[Theorem 3.2]{Petyt_Spriano_Unbounded} there exists a maximal family $\{\ov T_1,\ldots, \ov T_k\}\subset \ov\frakS$ of pairwise orthogonal domains with unbounded coordinate spaces, such that any other domain $R\in \ov\frakS$ with unbounded coordinate space is nested in one of the $\ov T_i$. Since $\ov U \transverse \ov V$, they must be nested in the same $\ov T_i$, call it $\ov T$. In particular, there exist representatives $T\in \ov T$, $U'\in \ov U$ and $V'\in \ov V$ such that $U',V'\nest T$. But $U'$ and $V'$ fill the surface, so $T$ must be the whole $S$. This shows that $\C \ov S=\CSN$ is unbounded, as required.\end{proof}

\noindent We are now ready to check the assumptions from Subsection \ref{subsec:hinge}:

\begin{prop}\label{lem:ass_satisfied}
    Let $\MCG/N$ satisfy Convention \ref{conv:surface-inherited}. Then good domains correspond to $N$-orbits of annuli, and Assumptions \eqref{ass1} - \eqref{ass6} hold.
\end{prop}

\begin{proof}
    We first recall some properties of the HHG structure for mapping class groups (see e.g. \cite[Theorem 11.1]{HHSII}). Firstly, if a subsurface $U$ belongs to a complete support set, then it is either an annulus, a $S_{0,4}$ or a $S_{1,1}$. Moreover, if $U$ has complexity one then every complete support set containing $U$ must also contain all its boundary annuli, and in particular $U$ is not good; on the other hand, annuli are very good, since every annulus can be obtained as the intersection of two \emph{pants decompositions}, that is, two maximal collections of pairwise disjoint annuli. Finally, the rank of $\MCG$ is precisely the complexity of $S$.

    Now let $\MCG/N$ satisfy the Convention. We subdivide the rest of the proof into a series of claims:
\begin{claim}\label{claim:orbit_of_css}
    The $N$-orbit of a complete support set for $\MCG$ is a complete support set for $\MCG/N$, and every complete support set for $\MCG/N$ is of this form. In particular, the rank of $\MCG/N$ is $\zeta(S)$.
\end{claim}
\begin{proof}[Proof of Claim \ref{claim:orbit_of_css}]
    Let $\{U_i\}_{i=1}^{\zeta(S)}$ be a complete support set for $\MCG$. Since $\zeta(S)\ge 2$, the set of all curves lying on $\bigcup_{i=1}^{\zeta(S)} U_i$ has diameter at most $2$ inside $\CS$, hence the subsurfaces must belong to pairwise distinct orbits $\{\ov  U_i\}_{i=1}^{\zeta(S)}$ by the assumption on the minimum translation length. Moreover, by Convention \ref{conv:surface-inherited} the orbits $\{\ov  U_i\}_{i=1}^{\zeta(S)}$ are again pairwise orthogonal, and have unbounded coordinate spaces since each $\C \ov  U_i$ is isometric to $\C U_i$. Hence the projection of a complete support set for $\MCG$ is a support set for $\MCG/N$ with the same number of domains, and in particular the rank $\nu$ of $\MCG/N$ is at least ${\zeta(S)}$.
     
    Conversely, given a complete support set $\mathcal T=\{\ov U_1,\ldots,\ov U_{\nu}\}\subset \ov\frakS$ for $\MCG/N$, by Lemma \ref{lem:css_lift} we can find pairwise disjoint representatives $U_i\in \ov  U_i$, proving that $\nu$ is at most $\zeta(S)$; hence $\nu=\zeta(S)$, and every complete support set for $\MCG/N$ lifts to a complete support set for $\MCG$.
\end{proof}

\begin{claim}\label{claim:good_dom_for_mcg/n}
    Let $\ov  U\in\ov \frakS$ be a domain belonging to a complete support set. Then $\ov U$ is good if and only if it is the orbit of an annulus.
\end{claim}

\begin{proof}[Proof of Claim \ref{claim:good_dom_for_mcg/n}]
If $\ov  U$ is the orbit of an annulus $U$, we can find two pants decompositions $\mathcal U$ and $\mathcal V$  whose intersection is $\{U\}$. Since all curves in $\mathcal U\cup\mathcal V$ lie in the ball of radius $1$ centred at $U$ inside $\CS$, the assumption on the minimum translation length implies that any two domains in $\mathcal U\cup\mathcal V$ lie in different $N$-orbits. Hence $\ov {\mathcal U}\cap\ov {\mathcal V}=\{\ov U\}$. Conversely, if $\ov  U$ is the class of a subsurface $U$ of complexity one, then every complete support set $\ov {\mathcal U}$ containing $\ov  U$ lifts to a complete support set containing $U$, which must also contain the boundary annuli of $U$. Hence $\ov {\mathcal U}$ must contain the $N$-orbit of the boundary annuli of $U$.
\end{proof}
 
    \noindent We finally check Assumptions \eqref{ass1} - \eqref{ass6}:
     
    \textbf{Assumption \eqref{ass1}:} Whenever $\ov  U$ belongs to a complete support set, $\C\ov  U$ is isometric to $\C U$ for any of its representatives $U$, and in particular has at least two points at infinity.
     
    \textbf{Assumption \eqref{ass2}:} For every two $\ov  U$, $\ov  V$ which lift to disjoint annuli $U,V$, we can find two pants decompositions $\mathcal U$ and $\mathcal V$ whose intersection is $\{U,V\}$. Moreover, any two annuli in $\mathcal U\cup\mathcal V$ belong to different $N$-orbits, because their core curves are disjoint, so $\{\ov  U,\ov  V\}$ is the intersection of $\ov {\mathcal U}$ and $\ov {\mathcal V}$.
     
    \textbf{Assumption \eqref{ass3}:} Let $\ov  U$ belong to a complete support set. If $\ov  U$ is not good, we can pick one of its lift $U$ and find two complete support sets $\mathcal U$ and $\mathcal V$ whose intersection is $\{U,V_1,\ldots,V_n\}$, where $\{V_1,\ldots,V_n\}$ are the boundary annuli of $U$. Now, all curves in $\mathcal U\cup\mathcal V$ lie in the ball of radius $1$ centred at any curve $x$ which lies in $U$. Therefore, by the assumption on the minimum translation length, any two domains in $\mathcal U\cup\mathcal V$ must belong to different $N$-orbits, and in turn this means that $\{\ov  U,\ov  V_1,\ldots,\ov  V_n\}$ is the intersection of $\ov {\mathcal U}$ and $\ov {\mathcal V}$.
     
    \textbf{Assumption \eqref{ass4}:} If $\ov  U$ is good then $\C \ov U$ is isometric to an annular curve graph, hence  $|\partial\C\ov  U|=2$.
     
    \textbf{Assumption \eqref{ass5}:} It is enough to show that, if $\ov x,\ov y\in\CSN^{(0)}$ have the same \emph{link} inside $\CSN$ (i.e. if they are adjacent to the same vertices), then they coincide. Indeed, the subgraph spanned by $\{\ov x,\ov y\}\cup\link_{\CSN}(\ov x)$ has diameter at most $2$, and therefore can be lifted by Lemma \ref{lem://_go_to_orth_or_//}. Thus we get two lifts $x,y$ such that $\link_{\CS}(x)$ is the lift of $\link_{\CSN}(\ov x)$, and therefore coincides with $\link_{\CS}(y)$. Hence we are left to prove that, if two curves $x,y$ have the same link in $\CS$, then they must coincide. This is clearly true if the curves are disjoint, as then $x\in\link(y)-\link(x)$. Otherwise, let $S'$ be a component of $S-{x}$ which is not a pair of pants (here we are using that $S$ has complexity at least $2$), and let $z$ be the subsurface projection of $y$ inside $S'$. By applying a partial pseudo-Anosov of $S'$ to $z$ we can find a curve $z'$ inside $S$ which crosses $z$, and therefore also $y$.
     
    \textbf{Assumption \eqref{ass6}:} As we saw before, every $\ov U$ which is the orbit of an annulus arises as the intersection of two complete support sets which come from two pants decompositions. Therefore good domains are very good.
\end{proof}

\noindent As a consequence of Proposition \ref{lem:automorphism_of_frakS_G}, combined with the description of good domains from Proposition \ref{lem:ass_satisfied}, we get:
\begin{cor}\label{cor:auto_CSN}
    Let $\MCG/N$ satisfy Convention \ref{conv:surface-inherited}. Every quasi-isometry $f\colon \mathcal{MCG}/N\mapsto \mathcal{MCG}/N$ induces an automorphism $f_{\CSN}$ of $\CSN$.
\end{cor}

\subsection{Quasi-isometric rigidity of large translation quotient}\label{subsec:qi_for_mcgn}
Recall that two groups $G$ and $H$ are \emph{weakly commensurable} if there exist two finite-index subgroups $G'\le G$ and $H'\le H$ and two finite normal subgroups $L\unlhd G'$ and $M\unlhd H'$ such that $G'/L\cong H'/M$.
\begin{thm}[Quasi-isometric rigidity]\label{thm:qirigid}
   Let $\MCG/N$ satisfy Convention \ref{conv:surface-inherited}. If a finitely generated group $G$ and $\MCG/N$ are quasi-isometric then they are weakly commensurable. 
\end{thm}
\noindent We shall derive Theorem \ref{thm:qirigid} from the following quantitative version of quasi-isometric rigidity:
\begin{prop}[{\cite[Lemma 8.22]{mangioni2023rigidity}, see also \cite[Section 10.4]{Schwa}}]\label{prop:qi_quantitativa}
    Let $G$ be a finitely generated group, equipped with a fixed word metric induced by a finite generating set. Suppose that, for every $L\ge0$, there exists $R\ge 0$ such that:
    \begin{enumerate}[(a)]
    \item\label{item:1} Every $(L,L)$-self-quasi-isometry of $G$ lies within distance $R$ from the left multiplication by some element of $G$; 
    \item\label{item:2} If a $(L,L)$-self-quasi-isometry of $G$ lies within finite distance from the identity, then it lies within distance $R$ from the identity.
\end{enumerate}
Then a finitely generated group $H$ which is quasi-isometric to $G$ is also weakly commensurable to $G$.
\end{prop}

\begin{proof}[Proof of Theorem \ref{thm:qirigid}] Fix a word metric on $\MCG$ with respect to a finite generating set, which induces a word metric on $\MCG/N$. We now check the requirements of Proposition \ref{prop:qi_quantitativa}.
     
    \textbf{\ref{item:1}} Let $f\colon \MCG/N\to \MCG/N$ be a $(L,L)$-quasi-isometry, and let $f_{\CSN}$ be the induced automorphism of $\CSN$. By the combinatorial rigidity results from Section \ref{sec:combrig}, $f_{\CSN}$ is induced by some element $\ov  g\in\MCG/N$. Moreover, by Lemma \ref{lem:flats_to_flats} there exist a constant $C$ such that, whenever $\{\ov  U_i\}_{i=1}^\nu$ is a maximal collection of pairwise orthogonal orbits of annuli, $f$ maps the standard flat $\mathcal F_{\{\ov  U_i\}}$ within Hausdorff distance at most $C$ from $\mathcal F_{\{f_{\CSN}(\ov  U_i)\}}=\mathcal F_{\{\ov  g(\ov  U_i)\}}$, which in turn is within Hausdorff distance at most $C$ from $\ov  g\left(\mathcal F_{\{\ov  U_i\}}\right)$. Therefore $\dist_{Haus}\left(f(\mathcal F_{\{\ov  U_i\}}),\ov  g(\mathcal F_{\{\ov  U_i\}})\right)\le 2C$.
     
    Now, to show that $f$ and $\ov  g$ uniformly coarsely coincide it is enough to say that, for every point $x\in\MCG/N$, there exist two standard flats $\mathcal F,\mathcal F'$, supported on orbits of annuli, whose coarse intersection is within uniformly bounded Hausdorff distance from $x$. Indeed, if this is the case then $f$ and $\ov  g$ should uniformly coarsely agree on the coarse intersection $\mathcal F\Tilde{\cap}\mathcal F'$, and therefore on $x$. In turn, since $\MCG/N$ acts transitively on itself and maps orbits of annuli to orbits of annuli, it is enough to exhibit a single pair $\mathcal F,\mathcal F'$ of standard flats, supported on orbits of annuli, whose coarse intersection is bounded. 
         
    In turn, as a consequence of Lemma \ref{lem:coarse_int_orthants}, if $\mathcal F$ is supported on $\{\ov  U_i\}_{i=1}^\nu$, $\mathcal F'$ is supported on $\{\ov  V_i\}_{i=1}^\nu$, and $\{\ov  U_i\}_{i=1}^\nu\cap \{\ov  V_i\}_{i=1}^\nu=\emptyset$, then the coarse intersection is bounded. Thus, we are left to find two disjoint collections $\{\ov  U_i\}_{i=1}^\nu$ and $\{\ov  V_i\}_{i=1}^\nu$ of pairwise orthogonal orbits of annuli. By Lemma \ref{lem:lift}, it is enough to find two pants decompositions $\{U_i\}_{i=1}^\nu$ and $\{V_i\}_{i=1}^\nu$ such that $\{U_i\}_{i=1}^\nu\cap\{V_i\}_{i=1}^\nu=\emptyset$ and the subgraph of $\CS$ spanned by the core curves of $\{U_i\}_{i=1}^\nu\cup\{V_i\}_{i=1}^\nu$ is contained in a closed ball of radius $2$. To find such annuli, start with any pants decomposition $\{U_i\}_{i=1}^\nu$, and then successively replace each $U_i$ with an annulus $V_i$ which does not intersect any $V_j$ for $j<i$ nor any $U_k$ for $k>i$. By construction, $V_1,\ldots, V_{\nu-1}$ are disjoint from $U_\nu$, so the core curves of $\{U_i\}_{i=1}^\nu\cup\{V_i\}_{i=1}^\nu$ lie in the closed ball of radius $2$ around the core of $U_\nu$. This proves Item \ref{item:1}.
     
    \textbf{\ref{item:2}} Let $f\colon \MCG/N\to \MCG/N$ be a $(L,L)$-quasi-isometry which lies within finite distance from the identity. By Item \ref{item:1} we know that $f$ lies within distance $R$ from the left multiplication by \emph{any} element $\ov g$ which induces $f_{\CSN}$. In turn $f_{\CSN}$ is induced by $f_{Hin}$, thus if we show that this map is the identity then $\ov g$ can be chosen to be the identity, and the corollary follows. Now, for every good hinge $\sigma$, $f_{Hin}(\sigma)$ was defined in Remark \ref{rem:f_hin} as the unique hinge such that $d_{Haus}(h_{f_{Hin}(\sigma)}, f(h_\sigma))<\infty$. But then, since $d_{Haus}(f(h_{\sigma}),h_\sigma)<\infty$ we must have that $f_{Hin}(\sigma)=\sigma$, that is, $f_{Hin}$ is the identity. 
\end{proof}

\section{Algebraic rigidity}\label{sec:algrig}
\noindent Here we show that, whenever $\MCG/N$ satisfies Convention \ref{conv:surface-inherited}, the automorphism group of $\MCG/N$ and its abstract commensurator are both isomorphic to $\MCG/N$, via the action of $\MCG/N$ on itself by conjugation. This is Theorem \ref{thmintro:alg} from the introduction, which is covered by Corollaries \ref{cor:outmcg} and \ref{cor:commmcg} below. The main result of this Section is the following:

\begin{thm}\label{thm:Out}
        Let $S$ be a surface of complexity at least $2$, excluding $S_{1,2}$ and $S_{2,0}$, and let $\MCG/N$ satisfy Convention \ref{conv:surface-inherited}. Then any isomorphism $\phi\colon H\to H'$ between finite index subgroups of $\MCG/N$ is the restriction of an inner automorphism.
\end{thm}
\noindent Leaving $S_{2,0}$ out is necessary, since $\Out(\mathcal{MCG}^{\pm}(S_{2,0}))\cong \mathbb{Z}_2 \oplus \mathbb{Z}_2$ \cite{mccarthy}. We expect that a suitable modification of our arguments, possibly involving the full power of the machinery from \cite{Commensurating}, could recover this result and extend it to large translation quotients, but we prefer not to pursue this direction for simplicity and brevity.

We first recall some definitions and a theorem from \cite{Commensurating}.

\begin{defn}
Two elements $h$ and $g$ of a group $G$ are \emph{commensurable}, and we write $h \stackrel{G}{\approx} g$, if there exist $m,n\in\mathbb{Z}\setminus\{0\}$, $k\in G$ such that $kg^m k^{-1}=h^n$ (that is, if they have non-trivial conjugate powers).
\end{defn}

\begin{defn}
If a group $G$ acts by isometries on a hyperbolic space $\mathcal{S}$, an element $g\in G$ is \emph{loxodromic} if for some $x\in\mathcal{S}$ the map $\mathbb{Z}\to \mathcal{S}$, $n\mapsto g^n(x)$ is a quasi-isometric embedding. In the same setting, an element $g\in G$ is \emph{weakly properly discontinuous}, or \emph{WPD}, if for every $\varepsilon>0$ and any $x\in \mathcal{S}$ there exists $N\ge 0$ such that
$$\left | \left\{h\in G \mid \max\left\{d_{ \mathcal{S}}\left(x, h(x)\right), d_{ \mathcal{S}}\left(g^N(x), hg^N(x)\right)\right\}\le \varepsilon\right\}\right |<\infty$$
We denote by $\mathcal{L}_{WPD}$ the set of loxodromic WPD elements.
\end{defn}
The following result is a special case of \cite[Theorem 7.1]{Commensurating}. Roughly speaking, the theorem says that if $H$ is a subgroup of $G$ and both act “interestingly enough” on some hyperbolic space, then any homomorphism $\phi:H\to G$ is either (the restriction of) an inner automorphism or it maps some loxodromic WPD to an element which is not commensurable to it.

\begin{thm}\label{7.1}
    Let $G$ be a group acting coboundedly and by isometries on a hyperbolic space $\mathcal{S}$, with loxodromic WPD elements. Let $H\le G$ be a non-virtually-cyclic 
    subgroup such that $H\cap\mathcal{L}_{WPD}\neq \emptyset$, and suppose that $H$ does not normalise any non-trivial finite subgroup of $G$. Let $\phi\colon H\to G$ be a homomorphism such that $\phi(h) \stackrel{G}{\approx} h$ for every $h\in H\cap \mathcal{L}_{WPD}$. Then $\phi$ is the restriction of an inner automorphism.
\end{thm}

\noindent Our proof of Theorem \ref{thm:Out} will be very similar to that of \cite[Theorem 9.1]{mangioni2023rigidity}, from which we now abstract a general statement, both for clarity and for future reference. Recall that a finitely generated group is \emph{acylindrically hyperbolic} if it is not virtually-cyclic and it acts coboundedly and by isometries on a hyperbolic space $\mathcal{S}$, with loxodromic WPD elements.

\begin{thm}[Algebraic rigidity from quasi-isometric rigidity]\label{thm:Alg_from_qi}
    Let $G$ be an acylindrically hyperbolic group. Suppose that $G$ has no non-trivial finite normal subgroups, and that every self-quasi-isometry of $G$ is within bounded distance from the left multiplication by some element of $G$. Then any isomorphism between finite index subgroups of $G$ is the restriction of an inner automorphism.
\end{thm}

\begin{proof}
We just need to verify that the hypotheses of Theorem \ref{7.1} are satisfied for any isomorphism $\phi\colon H \to H'$ between subgroups of finite index of $G$. By definition of acylindrical hyperbolicity, the hypotheses on the action are satisfied. Moreover, $G$ is not virtually-cyclic, and therefore neither is $H$ since its index is finite. Furthermore, in view of the general \cite[Lemma 9.6]{mangioni2023rigidity}, if $G$ has no non-trivial finite normal subgroup then $H$ does not normalise any non-trivial finite subgroup.

We are left to show that $\phi$ has the required commensurating property, that is, for every $h\in H\cap \mathcal{L}_{WPD}$ we have that  $\phi(h) \stackrel{G}{\approx} h$. Indeed, $\phi$ can be extended to a quasi-isometry $\Phi\colon G\to G$ (for example, by precomposing $\phi$ with any closest-point projection $G\to H$). Hence we can find an element $g\in G$ whose left-multiplication is uniformly close to $\Phi$. From here, one can argue as in \cite[Lemma 9.5]{mangioni2023rigidity}, which only uses some basic properties of Cayley graphs of finitely generated groups, to prove that $\phi(h) \stackrel{G}{\approx} h$ whenever $h\in H$ has infinite order, and in particular whenever $h$ is loxodromic WPD.
\end{proof}

\noindent Hence, in order to prove Theorem \ref{thm:Out}, we are left to verify the hypotheses of Theorem \ref{thm:Alg_from_qi} when $G=\MCG/N$ and the surface $S$ is as in Theorem \ref{thm:Out}. We know that $\MCG/N$ is quasi-isometrically rigid, by Theorem \ref{thm:qirigid}, so we just need to show that it is acylindrically hyperbolic (Lemma \ref{lem:acyl_hyp}) and that it has no non-trivial finite normal subgroups (Lemma \ref{lem:nofinite}).

\begin{lem}\label{lem:acyl_hyp}
    Let $\MCG/N$ satisfy Convention \ref{conv:surface-inherited}. Then $\MCG/N$ is acylindrically hyperbolic.
\end{lem}
\begin{proof}
Since $\CSN$ is the main coordinate space in the hierarchical structure for $\MCG/N$, by \cite[Corollary 14.4]{HHSI} it is enough to check that $\CSN$ is unbounded, which we proved in Lemma \ref{lem:CSN_unbounded}, and that $\MCG/N$ is non-virtually-cyclic. To see the latter, notice that, if $x,y\in\CS^{(0)}$ are disjoint curves, then the $\mathbb{Z}^2$ subgroup generated by the Dehn twists $T_x$ and $T_y$ around $x$ and $y$, respectively, must inject in the quotient, since every element in $\langle T_x, T_y\rangle$ fixes the curve $x$ and has therefore trivial translation length.
\end{proof}

\begin{lem}\label{lem:nofinite}
    Let $S$ be a surface of complexity at least $2$, excluding $S_{1,2}$ and $S_{2,0}$, and let $\MCG/N$ satisfy Convention \ref{conv:surface-inherited}.  Then $\MCG/N$ has no non-trivial finite normal subgroup.
\end{lem}
\begin{proof}
Let $K\unlhd \MCG/N$ be a finite normal subgroup, and towards a contradiction let $\ov f\in K$ be a non-trivial element. If $\ov f$ acts trivially on $\CSN$ then $\ov f$ is the identity, by the injectivity part of either Theorem \ref{thm:ivanov} or Theorem \ref{thm:ivanov_S5}. Thus let $\ov x\in\CSN^{(0)}$ be a vertex which is not fixed by $\ov f$. Let $x\in\CS^{(0)}$ be a representative for $\ov x$, let $f\in\MCG$ be a representative for $\ov f$, let $T_x\in \MCG$ be the Dehn Twist around $x$, and let $\ov{T_x}\in\MCG/N$ be its image in the quotient. Since $K$ is finite, we can find $m\in\mathbb{N}_{>0}$ such that $\ov f=(\ov{T_x})^{-m}\ov f (\ov{T_x})^m$. Hence, denoting by $T_{f(x)}$ the Dehn twist around $f(x)$, and by $\ov 1\in \MCG/N$ the identity element of the quotient, we have
$$\ov 1=(\ov{T_x})^{-m}\left(\ov f (\ov{T_x})^m\ov f^{-1}\right)=(\ov{T_x})^{-m}(\ov{T_{f(x)}})^{m},$$
where we used how Dehn twists behave under conjugation. In other words, we have that $T_x^{-m}T_{f(x)}^{m}\in N$. 

Now, if $\ov x$ and $\ov f(\ov x)$ are adjacent in $\CSN$, then we could have chosen representatives $x$ and $f$ such that $x$ and $f(x)$ are disjoint curves. Then $T_x^{-m}T_{f(x)}^{m}$ would be non-trivial but it could not be an element of $N$, since its translation length is zero. 

Hence suppose that $\ov x$ and $\ov f(\ov x)$ are not adjacent in $\CSN$. Let $U\in\frakS$ be the annulus with core curve $x$, let $\ov U\in\ov\frakS$ be its $N$-orbit, and let $\rho_{\ov U}^{\ov{f(U)}}\subset \C\ov U$ be the bounded subset given by the HHG structure. Notice that
$$\rho_{\ov U}^{\ov{f(U)}}=\rho_{\ov U}^{(\ov{T_x})^{-m}(\ov{T_{f(x)}})^{m}\ov{f(U)}}=\rho_{\ov U}^{(\ov{T_x})^{-m}\ov{f(U)}}=(\ov{T_x})^{-m}\left(\rho_{\ov U}^{\ov{f(U)}}\right),$$
where we used that $T_{f(x)}$ fixes $f(x)$ and that projections are equivariant in a HHG structure. This contradicts the fact that $\ov{T_x}$ acts loxodromically on the quasiline $\C\ov U$, since $T_x$ acts loxodromically on $\C U$ and the projection $\C U\to \C \ov U$ is a $\text{Stab}(U)$-equivariant isometry. 
\end{proof}

\noindent As a consequence of Theorem \ref{thm:Out}, we get that any automorphism of $\MCG/N$ is the conjugation by some element $\ov g\in\MCG/N$. The following lemma states that such $\ov g$ is also unique:

\begin{lem}\label{lem:aut_inj}
Let $S$ be a surface of complexity at least $2$, excluding $S_{1,2}$ and $S_{2,0}$, and let $\MCG/N$ satisfy Convention \ref{conv:surface-inherited}. Then $\MCG/N$ is centerless.
\end{lem}

\begin{proof}
The centre of an acylindrically hyperbolic group is a finite normal subgroup \cite[Corollary 7.2]{Osin}, so the lemma follows from Lemmas \ref{lem:acyl_hyp} and \ref{lem:nofinite}.
\end{proof}

\noindent Combining Theorem \ref{thm:Out} and Lemma \ref{lem:aut_inj} we get:

\begin{cor}\label{cor:outmcg}
Let $S$ be a surface of complexity at least $2$, excluding $S_{1,2}$ and $S_{2,0}$, and let $\MCG/N$ satisfy Convention \ref{conv:surface-inherited}. Then the map $\MCG/N\to \Aut(\MCG/N)$, sending each element $\ov g\in \MCG/N$ to the conjugation by $\ov g$, is an isomorphism. In particular, $\Out(\MCG/N)$ is trivial.
\end{cor}

\noindent Finally, we recall the definition of the \emph{abstract commensurator} of a group $G$. Consider the set of all isomorphisms $H\to H'$ between finite-index subgroups of $G$. Let $\Comm(G)$ be the quotient of this set by the following equivalence relation: two isomorphisms are identified if they coincide on a finite index subgroup of $G$. Then $\Comm(G)$ can be endowed with a group structure, induced by composition. Our final result shows that $\Comm(\MCG/N)$ is “the smallest possible”:

\begin{cor}\label{cor:commmcg}
Let $S$ be a surface of complexity at least $2$, excluding $S_{1,2}$ and $S_{2,0}$, and let $\MCG/N$ satisfy Convention \ref{conv:surface-inherited}. Then the map $\MCG/N\to \Comm(\MCG/N)$, sending each element $\ov g\in \MCG/N$ to the conjugation by $\ov g$, is an isomorphism.
\end{cor}

\begin{proof}
The map is surjective by Theorem \ref{thm:Out}. Towards injectivity, suppose that the conjugation by $\ov g$ is the identity on a finite-index subgroup $H$, that is, $\ov g$ commutes with $H$. Then $\langle \ov g\rangle \cap H$ is in the center of $H$, which is finite as $H$ is acylindrically hyperbolic. Hence $\langle \ov g\rangle$ is a finite subgroup of $\MCG/N$ normalised by $H$, and must therefore be trivial (again by \cite[Lemma 9.6]{mangioni2023rigidity}).
\end{proof}

\section{Examples of large translation quotients}\label{sec:examples}
\noindent We gather here some examples of large translation quotients satisfying Convention~\ref{conv:surface-inherited}, including random quotients.

\begin{ex}\label{ex:MCG}
    Definition \ref{defn:lrq} applies to the trivial subgroup $N=\{1\}$, thus $\MCG$ itself is a large translation quotient. Moreover, as we pointed out before, the HHG structure on $\MCG$ from \cite[Theorem 11.1]{HHSII} satisfies Convention~\ref{conv:surface-inherited}. 
\end{ex}

\noindent In view of the above example, Theorem \ref{thm:qirigid} recovers quasi-isometric rigidity of mapping class groups:
\begin{cor}[QI-rigidity of $\MCG$]\label{cor;qirig_mcg}
    Let $S$ be a surface of complexity at least $2$ which is not a torus with two punctures. If a finitely generated group $G$ and $\MCG$ are quasi-isometric then they are weakly commensurable.
\end{cor}

Moving to the next example, the following is \cite[Definition 6.1]{hhs_asdim}, which in turn builds on the notion of \emph{hyperbolically embedded} subgroups from \cite[Definition 2.1]{dgo}:
\begin{defn}[Hierarchically hyperbolically embedded subgroups]\label{ex:hhe_sgr}
    A finitely generated subgroup $H\le \MCG$ is \emph{hierarchically hyperbolically embedded} if there exists an infinite generating set $\mathcal{T}$ for $ \MCG$ such that 
    \begin{itemize}
        \item $\text{Cay}(\MCG,\mathcal{T})$ is quasi-isometric to the curve complex;
        \item $H\cap \mathcal{T}$ generates $H$;
        \item $H$ is hyperbolically embedded in $(G,\mathcal{T})$, meaning that $\Cay{\MCG}{H\cup \mathcal T}$ is hyperbolic and that $H$ is proper with respect to the metric $\hat\dist$ obtained from measuring the length of a shortest path in $\Cay{\MCG}{H\cup \mathcal T}$ with the property that any two consecutive vertices in $H$ can only be connected by edges in $\mathcal T$.
    \end{itemize} 
    By work of \cite{ABD}, $H$ is \emph{convex-cocompact} in the sense of \cite{Kent-Leininger}, meaning that orbit maps $H\to \CS$, with respect to any finite generating set, are quasi-isometric embeddings.
\end{defn}

    \noindent Lemma \ref{lem:fixing_hhg_for_hhe} describes conditions on a subgroup $M\le H$ ensuring that the quotient by the normal closure of $M$ satisfies Convention \ref{conv:surface-inherited}. The Lemma relies on \cite[Theorem 6.2]{hhs_asdim}, which in turn should be thought of a "Dehn filling" result for hierarchically hyperbolically embedded subgroups.

\begin{lem}\label{lem:fixing_hhg_for_hhe} Let $S$ be a surface of complexity at least $2$ which is not a torus with two punctures, and let $H\le \MCG$ be a finitely generated subgroup which is hierarchically hyperbolically embedded. There exists a finite set $F\subset H-\{1\}$ such that, if a normal subgroup $M\unlhd H$ avoids $F$ and $H/M$ is hyperbolic, then the quotient $\MCG/N$ by the normal closure of $M$ satisfies Convention \ref{conv:surface-inherited}. 
\end{lem}

\begin{proof}
Since orbit maps of $H$ to $\CS$, with respect to any finite generating set, are quasi-isometric embeddings, there exists $R\ge0$ with the following property: if we denote by $B(1,R)$ the ball of radius $R$ centred at the identity of $H$, every subgroup $M$ which avoids the finitely many elements of $B(1,R)-\{1\}$ must act on $\CS$ with minimum translation length at least $9$. 

Now we claim that, if $M$ avoids a possibly bigger finite set, then the minimum translation length of the normal closure $N=\langle\langle M\rangle \rangle$ is also at least $9$. This follows from results of \cite{dgo} on hyperbolically embedded subgroups, but we can give a proof based on the machinery from \cite{ClayMangahas}. There the authors take as input our setting and output a \emph{projection complex}, which we shall think of as the data of:
\begin{itemize}
    \item a graph $\mathcal P$, whose vertices correspond in our case to orbits of the form $gH(x_0)\subset \CS$ for all $g\in\MCG$ and for a fixed basepoint $x_0\in\CS^{(0)}$;
    \item a notion of projection between vertices of $\mathcal P$, which for us is given by projections onto the quasi-convex subsets $gH(x_0)\subset \CS$;
    \item an action $N\circlearrowleft\mathcal{P}$, such that the stabiliser of $gH(x_0)$ is precisely $gMg^{-1}$ (this follows from \cite[Theorem 2.14]{dgo}, where $N$ is described as an infinite free product of conjugates of $M$).
\end{itemize}

An example of the construction of the projection complex for the case when $M$ is generated by a pseudo-Anosov mapping class can be found in \cite[Section 8.1]{ClayMangahas}.
 
Now, for every $x\in\CS^{(0)}$ there exists $g\in\MCG$ such that the orbit $gH(x_0)$ is uniformly close to $x$, since $\bigcup_{g\in\MCG} gH=\MCG$ and the action of $\MCG$ on the curve graph is cofinite. Now let $n\in N-\{1\}$, and we want to say that $\dist_{\CS}(x,nx)\ge 8$. Notice that we can think of $gH(x_0)$ and $ngH(x_0)$ as vertices of $\mathcal P$. Therefore, by \cite[Proposition 3.2]{ClayMangahas} one of the following must hold:
\begin{itemize}
    \item $n$ fixes $gH(x_0)$, and therefore it belongs to $gMg^{-1}$. In particular, the translation length of $n$ is at least $9$, since translation length is preserved under conjugation.
    \item $n$ fixes $ngH(x_0)$, and for the same reason its translation length is at least $9$.
    \item There exists a third vertex of $\mathcal P$, that is, a third orbit $g'H(x_0)$, on which $gH(x_0)$ and $ngH(x_0)$ have distant projection. Now, the projection of $x$ is close to the projection of $gH(x_0)$, and similarly $n x$ projects close to $ngH(x_0)$. Therefore $x$ and $nx$ have distant projections to $g'H(x_0)$, and in turn this means that $x$ and $nx$ are far from each other in $\CS$.
\end{itemize}
Now, \cite[Theorem 6.2]{hhs_asdim} states that, if $M$ avoids a (possibly larger) finite set and $H/M$ is hyperbolic, then $\MCG/N$ admits a hierarchically hyperbolic structure with the following properties:
    \begin{itemize}
        \item The index set is $\mathfrak T=\ov\frakS \cup \{NgH\}_{g\in \MCG}$;
        \item Orthogonality and nesting in $\ov\frakS$ correspond to orthogonality and nesting between representatives;
        \item If $\ov U\in \ov\frakS$ is not the maximal element, then $\C\ov U$ is isometric to $\C U$ via the projection map, for any $U\in \ov U$;
        \item If $T\in\mathfrak T$ is of the form $NgH$ then $T$ is orthogonal to no other domain.
    \end{itemize}
    To refine this structure, one can apply the procedure from \cite[Theorem 3.7]{ABD} to remove the domains without orthogonals, and get a new structure $(\MCG/N, \mathfrak R)$ such that:
    \begin{itemize}
        \item The new index set coincides with the maximal element, together with all elements $\mathfrak T$ admitting an orthogonal domain with unbounded coordinate space. Hence $\mathfrak R=\ov \frakS$;
        \item Orthogonality and nesting are inherited from the original structure;
        \item If $\ov U\in \ov\frakS$ is not the maximal element, then $\C\ov U$ is unchanged;
        \item The top-level coordinate space is the space $\widehat X$ obtained from $\MCG/N$ after coning off the factors $F_{\ov U}$ for every non-maximal $\ov U\in \ov\frakS$. These factors correspond to the stabilisers of the action of $\MCG/N$ on $\CSN$. Therefore, by a version of the Milnor-\v{S}varc Lemma described in e.g. \cite[Theorem 5.1]{CC}, $\CSN$ is quasi-isometric to $\widehat X$. 
    \end{itemize}
   It is now clear that the above structure satisfies Convention \ref{conv:surface-inherited}, as required.
\end{proof}

\noindent As a special case of Lemma~\ref{lem:fixing_hhg_for_hhe}, we have that the quotient by a suitable power of a pseudo-Anosov element fits in our framework:
\begin{cor}\label{cor:pA_large}
    Let $S$ be a surface of complexity at least $2$ which is not a torus with two punctures, and let $f\in\MCG$ be pseudo-Anosov element. There exists $K_0\in\mathbb{N}$ such that, if $K\in\mathbb{Z}-\{0\}$ is a multiple of $K_0$, then the quotient $\MCG/\langle\langle f^K\rangle\rangle$ satisfies Convention \ref{conv:surface-inherited}.
\end{cor}

\begin{proof}
    The \emph{elementary closure} $E(f)$ of a pseudo-Anosov element $f\in\MCG$, that is, the unique maximal virtually-cyclic subgroup containing $f$, is hierarchically hyperbolically embedded (this ultimately follows from \cite[Theorem 6.50]{dgo}). Moreover  $\langle f\rangle $ has finite index in $E(f)$, therefore there exists $K_0$ such that $\langle f^{K_0}\rangle $ is normal in $E(f)$. We can further assume that $\langle f^{K_0}\rangle $ avoids the finite set $F\subset E(f)$ which appears in Lemma \ref{lem:fixing_hhg_for_hhe}, so that the hypotheses of the latter are satisfied by $\langle f^{K}\rangle $ whenever $K$ is a non-trivial multiple of $K_0$.
\end{proof}

\noindent We conclude this Section by showing that random quotients satisfy Convention \ref{conv:surface-inherited}. Then Theorems~\ref{thmintro:comb_for_random} to \ref{thmintro:alg_for_rand} will follow as special cases of Theorems~\ref{thmintro:comb} to \ref{thmintro:alg}. Let us first define which probability measures we allow on $\MCG$. 

\begin{defn}\label{defn:permissible}
    Let $G$ be a group acting acylindrically on a hyperbolic space $X$. Given a probability measure $\mu$ on $G$, define the \emph{support} of $\mu$ as $\Supp(\mu)=\{g\in G\mid \mu(g)>0\}$, and let $\Gamma_\mu$ be the semi-group generated by $\Supp(\mu)$. A probability measure $\mu$ is \emph{permissible} if it is:
    \begin{itemize}
        \item \textit{bounded}: The support of $\mu$ ha bounded orbits on $X$;
        \item \textit{reversible}: $\Gamma_\mu$ is a subgroup of $G$;
        \item \textit{non-elementary}: $\Gamma_\mu$ acts non-elementarily on $G$;
        \item  If we denote by $E(\mu)$ the maximal finite subgroup normalised by $\Gamma_\mu$, which exists by e.g. \cite[Lemma 5.5]{H16}, then $E(\mu)=E(G)$.
    \end{itemize}
\end{defn}

Given permissible probability measures $\mu_1,\ldots,\mu_k$ on $\MCG$, with respect to the action on $\CS$, let $\{X^1_n,\ldots,X^k_n\}_{n\in\mathbb{N}}$ be the associated random walks, and let $N_n=\langle\langle X^1_n,\ldots,X^k_n\rangle\rangle$. Call $\MCG/N_n$ a \emph{random quotient} of $\MCG$. Moreover, recall that a property holds \emph{asymptotically almost surely} (a.a.s.) if $\MCG/N_n$ satisfies the property with probability approaching $1$ as $n\to \infty$.

\begin{lem}\label{lem:random_are_LT}
    Let $S$ be a surface of complexity at least $2$ which is not a torus with two punctures. A random quotient of $\MCG$, with respect to permissible probability measures, satisfies Convention~\ref{conv:surface-inherited} asymptotically almost surely.
\end{lem}

\begin{proof} $\MCG/N_n$ is a.a.s. a large translation quotient by \cite[Remark 5.14]{randomquot_HHG}, and a HHG by \cite[Theorem A]{randomquot_HHG}. Furthermore, by the explicit description of the HHG structure from \cite[Construction 4.26 and Remark 4.28]{randomquot_HHG}, we see that $\MCG/N_n$ a.a.s. satisfies Convention~\ref{conv:surface-inherited}, as required.
\end{proof}

\section{Are injective self-maps isomorphisms?}\label{sec:questions}
\noindent The following question was asked by Jason Behrstock, during an exchange on a first draft of this paper:
\begin{questintro}\label{quest:injective}
Let $S$ be either a surface of complexity at least $2$, excluding $S_{2,0}$ and $S_{1,2}$. Does there exist $D\ge 0$ such that, if the minimum translation length of a normal subgroup $N$ is at least $D$, \emph{any injection} from a finite index subgroup $H\le \MCG/N$ to $\MCG/N$ is induced by an inner automorphism?
\end{questintro}
\noindent A positive answer would imply that a large translation quotient is \emph{co-Hopfian}, that is, every self-monomorphism is an isomorphism.

\begin{rem}
    Question~\ref{quest:injective} is a strengthening of Theorem \ref{thmintro:alg}, whose techniques, however, do not carry over. Indeed, the arguments in Section~\ref{sec:algrig} rely on Theorem \ref{thmintro:qi} to produce a candidate conjugation, so we need the injection $H\hookrightarrow \MCG/N$ to be a quasi-isometry; in turn, this is true if and only if the image has finite index. 
\end{rem}

\noindent Question \ref{quest:injective} has a positive answer for mapping class groups, by work of Irmak \cite{Irmak1, Irmak2, Irmak3} and Behrstock-Margalit \cite{Behrstock-Margalit}. The main strategy of all these papers is articulated in two steps:
\begin{itemize}
    \item First, one proves that any \emph{superinjective map} of $\CS$ (that is, any self-map preserving adjacency and non-adjacency between vertices) is induced by an element of $\MCG$, generalising Ivanov's theorem.
    \item Then one shows that any injection from a finite-index subgroup of $\MCG$ into $\MCG$ maps powers of Dehn Twists to powers of Dehn Twists, preserving commutation and non-commutation. Therefore the injection induces a superinjective map of $\CS$, and this produces a candidate element of $\MCG$ whose conjugation should restrict to the given injection.
\end{itemize}
Therefore the following sub-questions naturally arise:

\begin{questintro}
    In the setting of Question \ref{quest:injective}, is any superinjective map of $\CS/N$ induced by an element of $\MCG/N$? 
\end{questintro}

\noindent This would be a generalisation of Theorem \ref{thmintro:comb}, which I believe to be true. One could try to lift a superinjective map of $\CS/N$ to a superinjective map of $\CS$, and then conclude by the results of Irmak and Behrstock-Margalit.

\begin{questintro}
    In the setting of Question \ref{quest:injective}, does the injection $H\hookrightarrow \MCG/N$ map powers of Dehn Twists to powers of Dehn Twist?
\end{questintro}

\noindent Irmak's proof is ultimately a refined version of Ivanov's algebraic description of Dehn Twists from \cite{Ivanov_DT}, in terms of the centres of their centralisers in $\MCG$. It is possible that a similar characterisation holds in our quotients as well. 

\bibliography{biblio}

\newcommand{\etalchar}[1]{$^{#1}$}
\begin{thebibliography}{BKMM12}

\bibitem[ABD21]{ABD}
Carolyn Abbott, Jason Behrstock, and Matthew~Gentry Durham.
\newblock Largest acylindrical actions and stability in hierarchically hyperbolic groups.
\newblock {\em Trans. Amer. Math. Soc. Ser. B}, 8:66--104, 2021.
\newblock With an appendix by Daniel Berlyne and Jacob Russell.

\bibitem[ABM{\etalchar{+}}25]{randomquot_HHG}
Carolyn Abbott, Daniel Berlyne, Giorgio Mangioni, Thomas Ng, and Alexander~J. Rasmussen.
\newblock Spinning and random quotients preserve hierarchical hyperbolicity.
\newblock Preprint, {arXiv}:2507.16677 [math.{GR}], 2025.

\bibitem[AL13]{AL}
Javier Aramayona and Christopher~J. Leininger.
\newblock Finite rigid sets in curve complexes.
\newblock {\em J. Topol. Anal.}, 5(2):183--203, 2013.

\bibitem[AMS16]{Commensurating}
Yago Antol\'{\i}n, Ashot Minasyan, and Alessandro Sisto.
\newblock Commensurating endomorphisms of acylindrically hyperbolic groups and applications.
\newblock {\em Groups Geom. Dyn.}, 10(4):1149--1210, 2016.

\bibitem[BHMS24]{BHMS}
Jason Behrstock, Mark Hagen, Alexandre Martin, and Alessandro Sisto.
\newblock A combinatorial take on hierarchical hyperbolicity and applications to quotients of mapping class groups.
\newblock {\em J. Topol.}, 17(3):94, 2024.
\newblock Id/No e12351.

\bibitem[BHS17a]{hhs_asdim}
Jason Behrstock, Mark~F. Hagen, and Alessandro Sisto.
\newblock Asymptotic dimension and small-cancellation for hierarchically hyperbolic spaces and groups.
\newblock {\em Proc. Lond. Math. Soc. (3)}, 114(5):890--926, 2017.

\bibitem[BHS17b]{HHSI}
Jason Behrstock, Mark~F. Hagen, and Alessandro Sisto.
\newblock Hierarchically hyperbolic spaces, {I}: {C}urve complexes for cubical groups.
\newblock {\em Geom. Topol.}, 21(3):1731--1804, 2017.

\bibitem[BHS19]{HHSII}
Jason Behrstock, Mark~F. Hagen, and Alessandro Sisto.
\newblock Hierarchically hyperbolic spaces {II}: {C}ombination theorems and the distance formula.
\newblock {\em Pacific J. Math.}, 299(2):257--338, 2019.

\bibitem[BHS21]{quasiflat}
Jason Behrstock, Mark~F. Hagen, and Alessandro Sisto.
\newblock Quasiflats in hierarchically hyperbolic spaces.
\newblock {\em Duke Math. J.}, 170(5):909--996, 2021.

\bibitem[BKMM12]{BKMM}
Jason Behrstock, Bruce Kleiner, Yair Minsky, and Lee Mosher.
\newblock Geometry and rigidity of mapping class groups.
\newblock {\em Geom. Topol.}, 16(2):781--888, 2012.

\bibitem[BM06]{Behrstock-Margalit}
Jason Behrstock and Dan Margalit.
\newblock Curve complexes and finite index subgroups of mapping class groups.
\newblock {\em Geom. Dedicata}, 118:71--85, 2006.

\bibitem[BM19]{Brendle_Margalit}
Tara~E. Brendle and Dan Margalit.
\newblock Normal subgroups of mapping class groups and the metaconjecture of {I}vanov.
\newblock {\em J. Amer. Math. Soc.}, 32(4):1009--1070, 2019.

\bibitem[CC07]{CC}
Ruth Charney and John Crisp.
\newblock Relative hyperbolicity and {A}rtin groups.
\newblock {\em Geom. Dedicata}, 129:1--13, 2007.

\bibitem[CM22]{ClayMangahas}
Matt Clay and Johanna Mangahas.
\newblock Hyperbolic quotients of projection complexes.
\newblock {\em Groups Geom. Dyn.}, 16(1):225--246, 2022.

\bibitem[Dah18]{dahmani:rotating}
Fran\c{c}ois Dahmani.
\newblock The normal closure of big {D}ehn twists and plate spinning with rotating families.
\newblock {\em Geom. Topol.}, 22(7):4113--4144, 2018.

\bibitem[DDLS21]{VeechII}
Spencer Dowdall, Matthew~G. Durham, Christopher~J. Leininger, and Alessandro Sisto.
\newblock Extensions of veech groups ii: Hierarchical hyperbolicity and quasi-isometric rigidity, 2021.
\newblock \url{arxiv.org/abs/2111.00685}.

\bibitem[DGO17]{dgo}
F.~Dahmani, V.~Guirardel, and D.~Osin.
\newblock Hyperbolically embedded subgroups and rotating families in groups acting on hyperbolic spaces.
\newblock {\em Mem. Amer. Math. Soc.}, 245(1156):v+152, 2017.

\bibitem[DHS21]{dfdt}
Fran\c{c}ois Dahmani, Mark Hagen, and Alessandro Sisto.
\newblock Dehn filling {D}ehn twists.
\newblock {\em Proc. Roy. Soc. Edinburgh Sect. A}, 151(1):28--51, 2021.

\bibitem[FM12]{FarbMargalit}
Benson Farb and Dan Margalit.
\newblock {\em A primer on mapping class groups}, volume~49 of {\em Princeton Mathematical Series}.
\newblock Princeton University Press, Princeton, NJ, 2012.

\bibitem[Har86]{Harer:CS_wedge}
John~L. Harer.
\newblock The virtual cohomological dimension of the mapping class group of an orientable surface.
\newblock {\em Invent. Math.}, 84(1):157--176, 1986.

\bibitem[Hul16]{H16}
Michael Hull.
\newblock Small cancellation in acylindrically hyperbolic groups.
\newblock {\em Groups Geom. Dyn.}, 10(4):1077--1119, 2016.

\bibitem[Irm04]{Irmak2}
Elmas Irmak.
\newblock Superinjective simplicial maps of complexes of curves and injective homomorphisms of subgroups of mapping class groups.
\newblock {\em Topology}, 43(3):513--541, 2004.

\bibitem[Irm06a]{Irmak1}
Elmas Irmak.
\newblock Complexes of nonseparating curves and mapping class groups.
\newblock {\em Michigan Math. J.}, 54(1):81--110, 2006.

\bibitem[Irm06b]{Irmak3}
Elmas Irmak.
\newblock Superinjective simplicial maps of complexes of curves and injective homomorphisms of subgroups of mapping class groups. {II}.
\newblock {\em Topology Appl.}, 153(8):1309--1340, 2006.

\bibitem[Iva88]{Ivanov_DT}
N.~V. Ivanov.
\newblock Automorphisms of {T}eichm\"{u}ller modular groups.
\newblock In {\em Topology and geometry---{R}ohlin {S}eminar}, volume 1346 of {\em Lecture Notes in Math.}, pages 199--270. Springer, Berlin, 1988.

\bibitem[Iva97]{Ivanov:autC}
Nikolai~V. Ivanov.
\newblock Automorphisms of complexes of curves and of {T}eichm\"{u}ller spaces.
\newblock In {\em Progress in knot theory and related topics}, volume~56 of {\em Travaux en Cours}, pages 113--120. Hermann, Paris, 1997.

\bibitem[KL08]{Kent-Leininger}
Richard~P. Kent, IV and Christopher~J. Leininger.
\newblock Shadows of mapping class groups: capturing convex cocompactness.
\newblock {\em Geom. Funct. Anal.}, 18(4):1270--1325, 2008.

\bibitem[Kor99]{Korkmaz}
Mustafa Korkmaz.
\newblock Automorphisms of complexes of curves on punctured spheres and on punctured tori.
\newblock {\em Topology Appl.}, 95(2):85--111, 1999.

\bibitem[Luo00]{Luo}
Feng Luo.
\newblock Automorphisms of the complex of curves.
\newblock {\em Topology}, 39(2):283--298, 2000.

\bibitem[McC86]{mccarthy}
John~D. McCarthy.
\newblock Automorphisms of surface mapping class groups. {A} recent theorem of {N}. {Ivanov}.
\newblock {\em Invent. Math.}, 84:49--71, 1986.

\bibitem[McL19]{McLeay}
Alan McLeay.
\newblock Geometric normal subgroups in mapping class groups of punctured surfaces.
\newblock {\em New York J. Math.}, 25:839--888, 2019.

\bibitem[Min96]{Minsky_geom_approach}
Yair~N. Minsky.
\newblock A geometric approach to the complex of curves on a surface.
\newblock In {\em Topology and {T}eichm\"{u}ller spaces ({K}atinkulta, 1995)}, pages 149--158. World Sci. Publ., River Edge, NJ, 1996.

\bibitem[MM99]{MM_1}
Howard~A. Masur and Yair~N. Minsky.
\newblock Geometry of the complex of curves. {I}: {Hyperbolicity}.
\newblock {\em Invent. Math.}, 138(1):103--149, 1999.

\bibitem[MM00]{MM_2}
H.~A. Masur and Y.~N. Minsky.
\newblock Geometry of the complex of curves. {II}: {Hierarchical} structure.
\newblock {\em Geom. Funct. Anal.}, 10(4):902--974, 2000.

\bibitem[MS25]{mangioni2023rigidity}
Giorgio Mangioni and Alessandro Sisto.
\newblock Rigidity of mapping class groups mod powers of twists.
\newblock {\em Proceedings of the Royal Society of Edinburgh: Section A Mathematics}, page 1–71, 2025.

\bibitem[Osi16]{Osin}
D.~Osin.
\newblock Acylindrically hyperbolic groups.
\newblock {\em Trans. Am. Math. Soc.}, 368(2):851--888, 2016.

\bibitem[PS23]{Petyt_Spriano_Unbounded}
Harry Petyt and Davide Spriano.
\newblock Unbounded domains in hierarchically hyperbolic groups.
\newblock {\em Groups Geom. Dyn.}, 17(2):479--500, 2023.

\bibitem[Sch95]{Schwa}
Richard~Evan Schwartz.
\newblock The quasi-isometry classification of rank one lattices.
\newblock {\em Inst. Hautes \'{E}tudes Sci. Publ. Math.}, 82:133--168 (1996), 1995.

\bibitem[Sch20]{Schleimer}
Saul Schleimer.
\newblock Notes on the complex of curves, 2020.

\end{thebibliography}
\bibliographystyle{alpha}

\end{document}